\documentclass[10pt,letterpaper]{amsart} 

\usepackage{color}
\usepackage[letterpaper, margin=1in]{geometry}
\usepackage{subfig}
\usepackage{multicol}
\usepackage{listings}% http://ctan.org/pkg/listings
\usepackage{amsmath, amsthm, amssymb, wasysym, verbatim, bbm, xcolor, graphics}
\usepackage{url}
\usepackage{mathtools}
\usepackage[colorlinks,linkcolor=blue]{hyperref}

\usepackage{xfrac}
\usepackage{float}
\usepackage{todonotes}
\newcommand{\R}{\mathbb{R}}

\newcommand{\N}{\mathbb{N}}

\newcommand{\tu}{\widetilde{u}}
\newcommand{\hu}{\widehat{u}}

\newcommand\norm[1]{\left\lVert#1\right\rVert}
\newcommand{\Grad}{\nabla}
\newcommand{\Div}{\operatorname{div}}
\newcommand{\dom}{\Omega}
\newcommand{\Vh}{\mathbb{V}_h}
\newcommand{\Ph}{\mathbb{P}_h}

\newcommand{\Yh}{\mathbb{Y}_h}

\newcommand{\Uh}{\mathbb{U}_h}
\newcommand{\weak}{\rightharpoonup}

\newcommand{\weakstar}{\overset{\star}\rightharpoonup}
\newcommand{\PUh}{\mathcal{P}_{\Uh}}
\newcommand{\PPh}{\mathcal{P}_{\Ph}}

\newcommand{\PVh}{\mathcal{P}_{\Vh}}

\newcommand{\Pl}{\mathcal{P}} % Leray projector

\newtheorem{lemma}{Lemma}[section]

\newtheorem{theorem}[lemma]{Theorem}
\newtheorem{remark}[lemma]{Remark}

\newtheorem*{maintheorem*}{Main Theorem}
\theoremstyle{definition}{\newtheorem{definition}[lemma]{Definition}}

\allowdisplaybreaks

\numberwithin{equation}{section}

\title[Convergence of a BDF2 Projection Method]{Convergence of a Second-Order Projection Method to Leray-Hopf Solutions of the Incompressible Navier-Stokes Equations}
\date{\today}
\thanks{This work was supported in part by the U.S. Department of Energy, Office of Science, Office of Advanced Scientific Computing Research's Applied Mathematics Competitive Portfolios program under Contract No. AC02-05CH11231 and by National Science Foundation award DMS 2438083.}
\author[F. Weber]{Franziska Weber}
\address[Franziska Weber]{\newline Department of Mathematics \newline UC Berkeley \newline Evans Hall, Berkeley,  CA 94702, USA.}
\email[]{fweber@berkeley.edu}

\begin{document}
	% \pagenumbering{arabic}
	
	\begin{abstract}
		We analyze a second-order projection method for the incompressible Navier-Stokes equations on bounded Lipschitz domains. The scheme employs a Backward Differentiation Formula of order two (BDF2) for the time discretization, combined with conforming finite elements in space. Projection methods are widely used to enforce incompressibility, yet rigorous convergence results for possibly non-smooth solutions have so far been restricted to first-order schemes. We establish, for the first time, convergence (up to subsequence) of a second-order projection method to Leray–Hopf weak solutions under minimal assumptions on the data, namely $u_0 \in L^2_{\mathrm{div}}(\Omega)$ and $f \in L^2(0,T;L^2_{\mathrm{div}}(\Omega))$.
		
		Our analysis relies on two ingredients: A discrete energy inequality providing uniform $L^\infty(0,T;L^2(\Omega))$ and $L^2(0,T;H^1_0(\Omega))$ bounds for suitable interpolants of the discrete velocities, and a compactness argument combining Simon’s theorem with refined time-continuity estimates. These tools overcome the difficulty that only the projected velocity satisfies an approximate divergence-free condition, while the intermediate velocity is controlled in space.
		We conclude that a subsequence of the approximations converges to a Leray–Hopf weak solution. This result provides the first rigorous convergence proof for a higher-order projection method under no additional assumptions on the solution beyond those following from the standard a priori energy estimate.
	
	\end{abstract}
	
	\maketitle

	\section{Introduction}
	
	We consider the incompressible Navier–Stokes equations describing the motion of a homogeneous viscous fluid in a bounded domain $\dom \subset \R^d$, $d=2,3$, over a finite time interval $[0,T]$:
	\begin{equation}
		\label{eq:NS}
		\begin{split}
			\partial_t u + (u\cdot\Grad)u + \nabla p &= \mu \Delta u + f,\quad (t,x)\in [0,T]\times\dom,\\
			\Div u & = 0,\quad (t,x)\in [0,T]\times\dom,
		\end{split}
	\end{equation}
	where $u=u(t,x):[0,T]\times\Omega \to \mathbb{R}^d$
 is the fluid velocity, $p=p(t,x):[0,T]\times\Omega\to \mathbb{R}$
	is the pressure, and $f:[0,T]\times\Omega\to \mathbb{R}^d$
	 is a given forcing term. The domain $\Omega\subset\mathbb{R}^d$,
	$d=2,3$, is a bounded Lipschitz domain, $\mu>0$
	is the kinematic viscosity (inversely proportional to the Reynolds number), and $T>0$
is a fixed final time. We impose homogeneous Dirichlet boundary conditions on $u$ and prescribe initial data
	\begin{equation*}
		u(0,\cdot)=u_0\in L^2_{\text{div}}(\dom),
	\end{equation*}
	where $L^2_{\mathrm{div}}(\Omega)$
	denotes the subspace of square-integrable divergence-free vector fields with vanishing normal trace on $\partial\Omega$. We assume $f\in L^2(0,T;L^2_{\mathrm{div}}(\Omega))$, though the analysis can likely be extended to $f\in L^2(0,T;(H^1_{\mathrm{div}}(\Omega))^*)$ with appropriate treatment of the forcing approximation in the finite element space.
	
	One of the major challenges in the numerical approximation of this system is the coupling between the nonlinearity and the incompressibility constraint. A simple yet successful approach to address this difficulty is through splitting or fractional step methods. These methods first evolve equation~\eqref{eq:NS} in time to compute an intermediate velocity field that need not be divergence-free, then project this field onto the space of divergence-free functions. This approach was first introduced by Chorin~\cite{Chorin1967,Chorin1968} and Temam~\cite{Temam1969}, and is commonly referred to as the Chorin projection method or simply the projection method.
	The projection method has been extensively studied, with error estimates derived in various settings under smoothness assumptions on the velocity field. Such smoothness can be established for short times in 3D and globally in 2D under sufficient regularity assumptions on the initial data and forcing term~\cite{E2002,Hou1993,Hou1992,Achdou2000,Guermond1998,Shen1996}. For semi-discrete (time-discrete but spatially continuous) versions, convergence to weak Leray-Hopf solutions of the Navier-Stokes equations~\eqref{eq:NS} was established by Temam~\cite{Temam1977,Temam1968} under only the assumptions provided by a priori energy estimates, with convergence guaranteed up to a subsequence since uniqueness remains unknown in 3D.
	However, the convergence analysis for fully-discrete versions has proven more challenging and was only established recently. Kuroki and Maeda~\cite{Kuroki2020,Maeda2022} proved convergence for finite difference methods on arbitrary Lipschitz domains, while Gallou\"et et al.~\cite{Gallouet2023} established convergence for a finite volume method on rectangular domains. More recently, Eymard and Maltese~\cite{Eymard2024} and the author~\cite{chorinprojection} proved convergence for finite element spatial discretizations using conforming elements -- the former formulating the projection step as a Poisson problem and the latter as a Darcy problem. While finite element discretizations for projection methods have been extensively studied~\cite{Guermond1996,Guermond1998,Rannacher1992,Prohl1997}, convergence proofs for the general case with possibly non-smooth solutions were not available prior to~\cite{Eymard2024}.
	The primary technical challenge in these convergence proofs stems from the inability to directly apply the Aubin-Lions-Simon lemma, which is the standard tool for establishing precompactness of approximation schemes for the incompressible Navier-Stokes equations. This difficulty arises due to the presence of two velocity field approximations and insufficient bounds on the pressure variable. Specifically, while an $L^2(0,T;H^1_0(\Omega))$-bound is available for the intermediate velocity field $\tu$, it is not divergence free, and the time derivative cannot be bounded in the dual space of $H^1_0(\dom)\cap H^k(\dom)$ for some sufficiently large $k$ due to the lack of sufficiently strong estimates on the pressure variable. To overcome this obstacle, new compactness results were developed in~\cite{Eymard2024,chorinprojection} for the finite element setting.
	
	Despite these advances, all methods for which convergence has been established~\cite{Kuroki2020,Maeda2022,Eymard2024,Gallouet2023,chorinprojection} are limited to first-order accuracy. The question of whether approximations computed by second-order accurate projection methods converge to weak solutions of~\eqref{eq:NS} under minimal assumptions, specifically, requiring only $u_0\in L^2_{\operatorname{div}}(\Omega)$, remained open.
	
	\subsection{Contributions of this work} We address this gap by analyzing a second-order accurate incremental projection method~\cite{Goda1979} for~\eqref{eq:NS} using the second-order Backward Difference Formula (BDF2) for the time discretization, as in~\cite[Equations (3.3) and (3.4)]{Guermond2006}, and see also~\cite{Guermond1999}. Specifically, we let $N\in \N$ the number of time levels and   discretize the time interval $[0,T]$ into time levels $t^m=m\Delta t$, $m=0,1,\dots, N$ where $\Delta t = T/N$. At each time level $t^m$, we seek discrete approximations $u^m(x)$ and $p^m(x)$ to $u(t^m,x)$ and $ p(t^m,x)$. We denote by $\tu^{m+1}$ the intermediate velocity field computed in the first, prediction step. 
	We denote the initial data $(u^0, p^0)=(u_0,0)$. Then at every time level $t^m$, $m= 0,\dots, M-1$, given $(u^m,p^m)$, we update $(u^{m+1}, p^{m+1})$ according to the following two steps:
	\begin{enumerate}
		\item[{\bf Step 1}] (Prediction step): We compute $\tu^{m+1}$ through
		
		\begin{equation}\label{eq:step1semi}
			\frac{3\tu^{m+1}-4u^m+ u^{m-1}}{2\Delta t}+ (\hu^{m+1}\cdot\Grad)\tu^{m+1}+ \Grad p^m= \mu \Delta \tu^{m+1}+f^{m+1}
		\end{equation}
		with boundary condition
		\begin{equation*}
			\left.\tu^{m+1}\right|_{\partial\dom}=0,
		\end{equation*}
		and where we have defined
		\begin{align}
				\hu^{m+1} &= 2\tu^{m+1}-\tu^m,\label{eq:averages}\\
			f^m(x)&=\frac{1}{\Delta t}\int_{t^{m-1/2}}^{t^{m+1/2}}f(t,x)dt,\quad m=0,\dots, N,\quad t^{m\pm 1/2}:=t^m\pm\frac{\Delta t}{2}.\notag
		\end{align} 
		\item[{\bf Step 2}] (Projection step): Next we define $(u^{m+1},p^{m+1})$ via
		\begin{subequations}
			\label{eq:projection}
			\begin{align}
				3\frac{u^{m+1}-\tu^{m+1}}{2\Delta t }&= -\Grad (p^{m+1}-p^m),\\
				\Div u^{m+1} & = 0,
			\end{align}	
		\end{subequations}
		with boundary condition
		\begin{equation}
			\label{eq:bcstep2}
			\left.u^{m+1}\cdot n\right|_{\partial \dom}=0.
		\end{equation}
	\end{enumerate}
	If $f$ is not defined for $t>T$, we set $f\equiv 0$ for $t>T$. One can check that~\eqref{eq:averages} is a second order accurate extrapolation at time $t^{m+1}$.
	This algorithm is sometimes called the incremental pressure-correction scheme. 
	We combine this with an arbitrary-order finite element spatial discretization. Our spatial discretization requires certain stability properties of the corresponding $L^2$-orthogonal projection that are satisfied on quasi-uniform triangulations, the same conditions employed in~\cite{chorinprojection}.
	
	The key technical ingredients of our analysis include:
	\begin{itemize}
		\item \emph{Discrete energy estimates} that yield uniform a priori $L^\infty(0,T;L^2(\Omega))$ and $L^2(0,T;H^1_0(\Omega))$ bounds on appropriately defined interpolations of the computed approximations.
		\item \emph{Consistency of velocity approximations:} We establish that different interpolations based on the intermediate and final velocity fields converge to the same weak limit (up to subsequence).
		\item \emph{Strong convergence via compactness:} To handle the nonlinearity in~\eqref{eq:NS}, we prove that one of the approximating sequences has a strongly convergent subsequence in $L^2([0,T]\times \Omega)$ using Simon's compactness result~\cite{Simon1987}. The main technical challenge is deriving uniform time continuity estimates for the approximation sequence, which we achieve by combining techniques from~\cite{Eymard2024,chorinprojection} with our discrete energy estimates.
		\item \emph{Convergence to Leray-Hopf solutions:} Finally,  we prove that a subsequence of our approximations converges to a Leray-Hopf solution of the Navier-Stokes equations~\eqref{eq:NS}.
	\end{itemize}
	
	Other second order discretizations of the projection method besides~\eqref{eq:step1semi}--\eqref{eq:projection} are available: Firstly, instead of choosing a BDF2 time discretization, one could use a Crank-Nicolson type scheme, see e.g.~\cite{E1995}. One of the crucial advantages of using BDF2 in our case is that the convective term $(\hu^{m+1}\cdot\Grad)\tu^{m+1}$ allows for a spatial discretization with a skew symmetric trilinear form that vanishes when taking $\tu^{m+1}$ as a test function, and therefore allows for the previously mentioned discrete energy estimate. It is unclear whether this is possible for a Crank-Nicolson scheme with finite element spatial discretization. Secondly, one of the disadvantages of using the incremental formulation of the projection method chosen here is that it imposes an artificial Neumann boundary condition on the pressure which leads to lower provable convergence rates (in the smooth setting) for the pressure variables~\cite{Shen1996,E1995,Guermond1999,Prohl2008,Carelli2009}. This can be overcome by using the scheme proposed by Kim and Moin~\cite{Kim1985} or the one proposed by Orszag, Israeli and Deville~\cite{Orszag1986}. The disadvantage of those schemes is that the boundary conditions make them inconvenient for a finit element discretization and therefore often spatial discretizations using finite differences or spectral methods are chosen instead for those schemes~\cite{Guermond2006}. For these reasons, we chose to analyze the incremental scheme using BDF2 time discretization~\eqref{eq:step1semi}--\eqref{eq:projection}.
	In particular, we prove, for the first time, convergence of a second-order accurate projection method for the incompressible Navier-Stokes equations to Leray-Hopf weak solutions, without assuming any additional regularity of the exact solution beyond the classical energy bounds. This shows that higher-order accuracy can be achieved without sacrificing rigorous convergence guarantees such as those already available for first order accurate methods.

	\subsection{Organization} The remainder of this article is structured as follows: Section~\ref{sec:prelim} introduces the necessary notation and the definition of weak solutions. Section~\ref{sec:num} presents the fully discrete numerical scheme for~\eqref{eq:NS} and establishes a priori discrete energy estimates. Section~\ref{sec:convergence} proves convergence of the scheme to a Leray-Hopf weak solution. The Appendix contains auxiliary technical results.
	
	\section{Preliminaries}\label{sec:prelim}
	In this section, we will introduce notation that we will use frequently later on, and  the definition of weak solutions for~\eqref{eq:NS} which are called Leray-Hopf solutions after Leray~\cite{Leray1934} and Hopf~\cite{Hopf1951} who first proved their existence.
	
	\subsection{Notation}

	For a Banach space $X$, we will denote its norm by  $\|\cdot\|_X$ and its dual space by $X^*$.  We will denote $L^p$ spaces (for example, $L^2(\dom)$ for square integrable functions defined over $\dom$), Sobolev spaces and Bochner spaces in standard ways, and will not distinguish between scalar, vector-valued and tensor-valued function spaces when it is clear from the context. In particular, we use $L^p(0,T; X)$ to denote the space of functions $f:[0,T]\to X$ which are $L^p$-integrable in the time variable $t\in [0,T]$ and take values in   $X$.  The inner product on $L^2(\dom)$ will be denoted by $(\cdot, \cdot)$.
	For a vector field $u:\dom\to\R^d$, we denote  the divergence by $\Div u = \sum_{j=1}^d \partial_j u^j$ and the gradient by $\Grad u = (\partial_i u^j)_{i,j=1}^d$. 
	We will use the subscript $\Div$  to indicate vector spaces containing divergence-free functions, in particular, we will use 
	\begin{equation*}
		\label{eq:sigma_space}
		\begin{aligned}
			&C_{c,\Div}^\infty(\dom)=\{{\phi}\in C_c^\infty(\dom); \Div {\phi}=0\},\\
			&L^2_{\Div}(\dom)=\{ {\phi}\in L^2(\dom):\Div  {\phi}=0,  {\phi}\cdot {n}|_{\partial\dom}=0\}=\overline{C_{c,\Div}^\infty(\dom)}^{L^2(\dom)},\\
			&H^1_{ \Div}(\dom)=H^1_0(\dom)\cap L^2_{\Div}(\dom),
		\end{aligned}
	\end{equation*}
	where we used $n$ to denote the outward normal vector of the domain $\dom$. We will use
	\begin{equation*}
		L^2_0(\dom)= \{\phi\in L^2(\dom); \, \int_{\dom} \phi dx =0\}
	\end{equation*}
to denote the space of $L^2$-functions that have zero average over the domain.
	We will also use the Leray projector $\Pl: L^2(\dom)\to L^2_{\Div}(\dom)$, which is an orthogonal projection induced by the Helmholtz-Hodge decomposition $ {f}=\nabla g+ {h}$ for any $ {f}\in L^2(\dom)$~\cite{Temam1977}. Here, $g\in H^1(\dom)$ is a scalar field, and $ {h}\in L^2_{\Div}(\dom)$ is a divergence-free vector field. It can be proved that this decomposition is unique in $L^2(\dom)$. Then for any $ {f}\in L^2(\dom)$, it holds that $\mathcal{P}{f}={h}$.  
	
	It remains to define  Leray-Hopf solutions for the incompressible Navier-Stokes equations which were introduced by Leray~\cite{Leray1934} for $\dom=\R^d$ and Hopf~\cite{Hopf1951} for bounded domains. 
	\begin{definition}[Leray-Hopf solutions of~\eqref{eq:NS}]
		\label{def:weaksol}
		Assume that $u_0\in L^2_{\Div}(\dom)$ and $f\in L^2(0,T;L^2_{\Div}(\dom))$.
		Let $u:[0,T]\times\dom\to \R^d$ be a vector field that is divergence-free, i.e., $\Div u=0$ for almost every $(t,x)\in [0,T]\times\dom$ and that satisfies
		\begin{equation}
			\label{eq:regularity}
			\begin{split}
				&u\in L^\infty(0,T;L^2_{\Div}(\dom))\cap L^2(0,T;H^1_{\Div}(\dom)), \\
				&\partial_t u \in L^{4/3}(0,T;(H^1_{\Div}(\dom))^*), 
			\end{split}
		\end{equation}
		and 
		\begin{equation}
			\label{eq:initdata}
			u(0,x) = u_0(x),\quad \text{a.e. }\, (t,x)\in [0,T]\times\dom.
		\end{equation}
		In addition, let  $u$ satisfy 
		the distributional version of~\eqref{eq:NS}:
		\begin{equation}
			\label{eq:weakformu}
			\int_0^T  ({u},\partial_t v) dt +\int_0^T\int_\dom ((u\cdot\Grad )v) \cdot u dxdt
			=-\int_0^T (f,v) dt+\mu\int_0^T  (\Grad u,\Grad v )dt,
		\end{equation}	
		for test functions $v\in C_c^\infty((0,T);C^\infty_{c,\Div}(\dom))$, and the energy inequality:
		\begin{equation}
			\label{eq:energyineq}
			\frac12\norm{u(t)}_{L^2(\dom)}^2  +\mu\int_0^t\norm{\Grad u(s)}_{L^2}^2 ds
			\leq \frac12\norm{u_0}_{L^2(\dom)}^2 + \int_0^t (f,u) ds,
		\end{equation}
		for a.e. $t\in [0,T]$. Finally, we assume that $u$ is continuous at zero in $L^2(\dom)$, i.e.,
		\begin{equation*}
			\lim_{t\to 0} \norm{u(t)-u_0}_{L^2(\dom)} = 0.
		\end{equation*}
		Then we call $u$ a \emph{Leray-Hopf solution of~\eqref{eq:NS}}.
	\end{definition}

	\section{Numerical scheme}\label{sec:num}
	The scheme is based on a linearly implicit discretization of the projection method introduced by Chorin and Temam~\cite{Chorin1967,Chorin1968,Temam1969,Temam1977}. 
	The idea of the projection method (or fractional step method) for the incompressible Navier-Stokes equations is to split the evolution of the Navier-Stokes equations~\eqref{eq:NS} into two steps: In the first step the velocity field is evolved according to the convection and the dissipation terms which may violate the divergence constraint. An intermediate velocity field $\tu$ is computed. In the second step the intermediate velocity field is projected onto divergence free fields. 

As we already stated the time discretization in~\eqref{eq:step1semi}--\eqref{eq:projection}, we continue to describe the spatial discretization.

			\subsection{Spatial discretization}\label{sec:spatial}
			We let $\mathcal{T}_h=\{K\}$ be families of conforming quasi-uniform triangulations of $\dom$ made of simplices with mesh size $h>0$. Here $h=\max_{K\in \mathcal{T}_h}\text{diam}(K)$. For simplicity, we assume that $\dom=\dom_h$, so that there is no geometric error caused by the domain approximation. We let $\Uh\subset H^1_0(\dom)^d$, and $\Ph\subset H^1(\dom)\cap L^2_0(\dom)$  be finite dimensional subspaces scaled by a meshsize $h>0$, where we use $\Uh$ as the space in which we seek the intermediate velocity $\tu^{m+1}_h$ and $\Ph$ as the space in which we seek the approximation of the pressure $p^{m+1}_h$. The final velocity $u^{m+1}_h$ is sought in the space $\Yh = \Uh + \Grad \Ph$. Note that since $\Ph\subset H^1(\dom)$, $\Grad \Ph$ is well-defined, though it may contain discontinuous functions.  Also note that the final velocity $u_h^{m+1}$ may not satisfy homogeneous Dirichlet boundary conditions.
			The finite element spaces $\Uh$ and $\Ph$ need to satisfy the following  requirements: We require the $L^2$-orthgonal projections 
			$\PUh:L^2(\dom)^d\to \Uh$, $\PPh: L^2(\dom)\to \Ph$ defined by
			\begin{equation}\label{eq:L2projection}
				(u,\PUh v) = (u,v),\quad \forall \, u\in \Uh, \quad 	(p,\PPh q) = (p,q),\quad \forall \, p\in \Ph, 
			\end{equation}
			to be stable on $H^1$ and $L^r$, $1\leq r\leq \infty$, and have  good approximation properties. In particular, we need for some constant $C>0$, independent of the mesh size $h>0$,
			\begin{subequations}
				\label{eq:L2projproperties}
				\begin{align}
					\norm{\PUh v}_{L^r(\dom)}\leq C\norm{v}_{L^r(\dom)},&\quad \forall v \in L^r(\dom),\, r\in [1,\infty],\label{eq:Lr}\\
					\left|\PUh v\right|_{H^1(\dom)}\leq C\left|v\right|_{H^1(\dom)},&\quad \forall v\in H^1(\dom),\label{eq:H1}\\
					\norm{\PUh v - v}_{L^2(\dom)}\leq C h^s \norm{v}_{H^s(\dom)},&\quad \forall v\in H^s(\dom),\, s\in \left(\frac12,k+1\right],\label{eq:L2approx}\\
					\left|v-\PUh v\right|_{H^1(\dom)}\leq C h^{s-1} |v|_{H^s(\dom)},&\quad \forall v\in H^s(\dom),\, s\in [1,k+1],\label{eq:H1approx}
				\end{align}
			\end{subequations}
			where $k\in \N$ is the order of the finite element space, and we require the same properties for the projection $\PPh$. Under the assumption of quasi-uniformity of the meshes,~\eqref{eq:L2projproperties} can be achieved for finite elements based on piecewise polynomial spaces such as
			\begin{equation*}
				\Uh = \{v\in H^1_0(\dom)\, | \, \left.v\right|_K\in P_k(K),\, \forall \, K\in \mathcal{T}_h \},
			\end{equation*}
			where $P_k(K)$ is the space of polynomials of degreee $\leq k\in \N$.  See~\cite{Douglas1975} for a proof of~\eqref{eq:Lr},~\cite[Prop 22.19]{Ern2021} for a proof of~\eqref{eq:L2approx} and~\cite[Prop. 22.21]{Ern2021} for a proof of~\eqref{eq:H1} and~\eqref{eq:H1approx}. The condition of quasi-uniformity can be weakend in some cases, see Remark 22.23 in~\cite{Ern2021}.
			
			As in~\cite{Eymard2024}, we define the space $\Vh$
			\begin{equation}
				\label{eq:spaceV}
				\Vh = \left\{v\in L^2(\dom)\, ,\, (v,\Grad q) = 0,\quad \forall q\in \Ph \right\}.
			\end{equation}
			Note that this is not a finite element space and it is not finite dimensional. We clearly have $L^2_{\Div}(\dom)\subset \Vh$. We denote the orthogonal projection from $L^2(\dom)$ onto $\Vh$ by $\PVh:L^2(\dom)\to \Vh$, i.e., we have for any $v \in L^2(\dom)$ that
			\begin{equation*}
				(\PVh v,w) = (v,w),\quad \forall w\in \Vh.
			\end{equation*}
			Moreover, $\PVh$ is the identity on $\Vh$.
			Hence,
			\begin{equation*}
				\norm{\PVh v - v}_{L^2(\dom)}^2 = \norm{\PVh v }_{L^2(\dom)}^2 + \norm{v}_{L^2(\dom)}^2 - 2(\PVh v,v) = \norm{v}_{L^2(\dom)}^2 -  \norm{\PVh v }_{L^2(\dom)}^2,
			\end{equation*}
			and it follows that
			\begin{equation*}
				\norm{\PVh v}_{L^2(\dom)}\leq \norm{v}_{L^2(\dom)}. 
			\end{equation*}
			We also have the following simple `commuting' property for $\PVh$ and $\Pl: L^2(\dom)\to L^2_{\Div}(\dom)$, the Leray projector, proved in \cite[Lemma 4.4]{chorinprojection}:
			\begin{lemma}
				\label{lem:commuting}
				We have that
				\begin{equation}\label{eq:commuting}
					\Pl \PVh v = \PVh \Pl v = \Pl v,\quad \forall \, v\in L^2(\dom).
				\end{equation}
			\end{lemma}
			Next, we define the skew-symmetric trilinear form $b$ for functions $u,v,w\in H^1(\dom)$ with $u\cdot n=0$ on $\partial\dom$ as
			\begin{equation}
				\label{eq:defb}
				b(u,v,w) = \int_{\dom} (u\cdot\Grad) v \cdot w dx +\frac12 \int_{\dom} \Div u v\cdot w dx.
			\end{equation}
			We observe that $b(u,w,v)= - b(u,v,w)$ and hence $b(u,v,v)=0$ for $u,v\in H^1(\dom)$ with $u\cdot n =0$ on $\partial\dom$ or $v=0$ on $\partial\dom$.  
			
			We start by approximating the initial data for $\tu$: Given $u_0\in L^2_{\Div}(\dom)$, we let
			\begin{equation}
				\label{eq:initdataapprox}
				\tu^0_h = \PUh u_0, 
			\end{equation} 
			where $\PUh$ is the $L^2$-orthogonal projection.
			Since $\tu_h^0$ is not necessarily divergence free, we project it onto $\Vh$ using the following projection step: Find $(u_h^0,p_h^0)\in \Yh\cap \Vh\times\Ph$ such that
			\begin{align}\label{eq:zerothstep}\begin{split}					
				\left(\frac{u^{0}_h-\tu^{0}_h}{\Delta t }, v\right)&= -(\Grad p^{0}_h, v),\\
				\left( u^{0}_h,\Grad q\right)&=0,
				\end{split}	
			\end{align}
			for all $(v,q)\in \Yh\times \Ph$.
			
			Then we compute approximations using the following iterative scheme: 
			\begin{enumerate}
				\item[{\bf Step 1}] (Prediction step):
				For any $m\geq 1$, given $u_h^{m},u^{m-1}_h \in \Yh$, $\tu_h^m,\tu_h^{m-1}\in \Uh$ and $p_h^m\in\Ph $, find $\tu_h^{m+1} \in \Uh $, such that for all $v\in \Uh $,
					\begin{equation}
						\label{eq:step1fully}
						\left(\frac{3\tu^{m+1}_h-4u^m_h+ u^{m-1}_h}{2\Delta t},v\right)+ b(\hu_h^{m+1},\tu^{m+1}_h,v)-( p^m_h,\Div v)+\mu(\Grad \tu_h^{m+1},\Grad v)= (\PUh f^{m+1}, v),
					\end{equation}
					where we denoted $\hu_h^{m+1} = 2\tu^{m+1}_h - \tu_h^m$ and
					\begin{equation*}
						f^{m+1} : = \frac{1}{\Delta t}\int_{t^{m+1/2}}^{t^{m+3/2}} f(s)ds.
					\end{equation*}
					\item[{\bf Step 2}] (Projection step): Next we seek $(u^{m+1}_h,p^{m+1}_h)\in \Yh\times\Ph$ which satisfy for all $(v,q)\in \Yh\times\Ph$,
					\begin{subequations}
						\label{eq:projectionfullydiscrete}
						\begin{align}
							\label{eq:projection1}
							\left(3\frac{u^{m+1}_h-\tu^{m+1}_h}{2\Delta t }, v\right)&= -( \Grad (p^{m+1}_h-p^m_h), v),\\
							\label{eq:projection2}
							\left(u^{m+1}_h,\Grad q\right)&=0.
						\end{align}	
					\end{subequations}
				\end{enumerate}
				\begin{remark}\label{rem:1ststep}
					Since for the first step at $m=0$, the approximations at $m-1=-1$ are not defined, one can either use extrapolation to define $u$ at $t=-\Delta t$, or use a first order method to obtain $\tu^1_h,u^1_h,p^1_h$ as follows, as it was done for example in~\cite[Section 3]{Guermond2006}:
					First compute $\tu^1_h\in \Uh$ by
					\begin{equation}
						\label{eq:step1fully1step}
						%\label{eq:udisc}
						\left(\frac{\tu^{1}_h-u^0_h}{\Delta t},v\right)+ b(\tu_h^{0},\tu^{1}_h,v)-( p^0_h,\Div v)+\mu(\Grad \tu_h^{1},\Grad v)= (\PUh f^{1}, v),\quad \forall v\in \Uh.
					\end{equation}
					Then obtain $u^1_h\in \Yh$ via the following projection step
					\begin{subequations}
						\label{eq:projectionfullydiscretestep1}
						\begin{align}
							\label{eq:projection1step1}
							\left(\frac{u^{1}_h-\tu^{1}_h}{\Delta t }, v\right)&= -( \Grad (p^{1}_h-p^0_h), v),\quad \forall v\in \Yh\\%,\\
							\label{eq:projection2step1}
							\left(u^{1}_h,\Grad q\right)&=0,	\quad \forall q\in \Ph.
						\end{align}	
					\end{subequations}
					In the following analysis, we will assume that this approach has been taken, though we believe that it can be adapted to different approaches for the first time step. For example, one could alternatively define $\tu^{-1} = u_0+\Delta t((u_0\cdot \Grad ) u_0 - \mu \Delta u_0)$ and then $u_h^{-1}$ by
					\begin{equation*}
						\begin{split}
							(\tu^{-1}-u_h^{-1},v)+\Delta t (\Grad p^{-1}_h,v)&= 0,\quad \forall \, v\in \Yh,\\
							(u_h^{-1},\Grad q)& = 0,\quad \forall\, q\in \Ph,
						\end{split}
					\end{equation*}
					and then do the first time step using~\eqref{eq:step1fully}--\eqref{eq:projectionfullydiscrete}. However, this strategy requires $u_0\in H^2(\dom)$.
			\end{remark}
					\begin{remark}\label{rem:PVhonUh}
					We note that restricted to $\Yh$, $\PVh$ corresponds exactly to the projection onto $\mathbb{H}_h= \ker(C_h)$, where the operator $C_h:\Yh\to \Ph$ is defined as in~\cite{Guermond1998}, by $(C_h u,q)=(u,\Grad q)$ for all $q\in \Ph$, i.e., one writes~\eqref{eq:projection1}--\eqref{eq:projection2} as
					\begin{equation*}
						\begin{split}
							3\frac{u^{m+1}_h-\tu^{m+1}_h}{2\Delta t } & = -C_h^\top (p^{m+1}_h-p_h^m),\\
							C_h u^{m+1}_h & = 0.
						\end{split}
					\end{equation*}
					Thus $\PVh$ restricted to $\Yh$ is exactly the projection step~\eqref{eq:projectionfullydiscrete} of the numerical scheme. In particular,  $\PVh \tu_h^{m+1} = u^{m+1}_h$. 
				\end{remark}
				\begin{remark}[Formulation of projection step as a Poisson problem]
					The projection step can be rewritten and solved as a Poisson problem as follows: Taking $\Grad q\in \Grad\Ph$ as a test function in~\eqref{eq:projection1} and using the weak divergence constraint~\eqref{eq:projection2}, we see that $p_h^{m+1}\in \Ph$ can be computed from
					\begin{equation}\label{eq:poisson}
						(\Grad p^{m+1}_h,\Grad q) = (\Grad p^m_h,\Grad q)-\left(\frac{3\Div \tu^{m+1}_h}{2\Delta t},q\right),
					\end{equation}
					and then $u_h^{m+1}$ can be computed via~\eqref{eq:projection1}.
				\end{remark}
				
				\subsection{Energy stability}
				Before showing that the scheme~\eqref{eq:step1fully}--\eqref{eq:projectionfullydiscrete} has a unique solution for every time step $m$, i.e., it is solvable, we show that any solution of the scheme satisfies a discrete energy bound. This is necessary to obtain stability of the scheme and to derive a priori estimates on the approximations that allow to pass to the limit in the mesh size $h$ and the time step $\Delta t$ and show convergence of the scheme. 
					\begin{lemma}[Discrete energy estimate]
					\label{lem:discenergy}
					The approximations computed by the scheme~\eqref{eq:step1fully}--\eqref{eq:projectionfullydiscrete} satisfy for any integer $M\in \left\{1,2,\dots, \frac{T}{\Delta t}\right\}$
					\begin{equation}
						\label{eq:energybalance}
						\begin{split}
						&E^M_h  	+\sum_{m=1}^{M-1}\norm{u^{m+1}_h-2u^m_h+ u^{m-1}_h}_{L^2}^2+ 2\sum_{m=1}^{M-1}\norm{\tu^{m+1}_h-u^{m+1}_h}_{L^2}^2 \\
						&\qquad 
						   + 4 \mu\Delta t \sum_{m=0}^{M-1}\norm{\Grad \tu^{m+1}_h}_{L^2}^2+3\norm{\tu^1_h-u^0_h}_{L^2}^2\\
						   &						   	\leq Ce^{3M\Delta t}\left(\norm{u_0}_{L^2}^2+ \Delta t\sum_{m=0}^{M-1}\norm{f^{m+1}}_{L^2}^2\right),
\end{split}
					\end{equation}
					where 
					\begin{equation*}
						E_h^M = \norm{u^{M}_h}_{L^2}^2 + \norm{2 u^M_h - u^{M-1}_h}_{L^2}^2+ \frac{4\Delta t^2}{3}\norm{\Grad p^{M}_h}_{L^2}^2 .
					\end{equation*}
				\end{lemma}
				
				\begin{proof}
				We start by observing that the following identities holds:
				\begin{align*}
					&2(3\tu^{m+1}_h -4u^m_h +u^{m-1}_h,\tu^{m+1}_h)\\
					& = 2(3u^{m+1}_h -4u^m_h +u^{m-1}_h,\tu^{m+1}_h) + 6(\tu^{m+1}_h-u^{m+1}_h,\tu^{m+1}_h)\\
					& = 2(3u^{m+1}_h -4u^m_h +u^{m-1}_h,\tu^{m+1}_h-u^{m+1}_h) +2(3u^{m+1}_h -4u^m_h +u^{m-1}_h,u^{m+1}_h) + 6(\tu^{m+1}_h-u^{m+1}_h,\tu^{m+1}_h).
				\end{align*}
				Taking $v=3u^{m+1}_h-4 u^m_h+u^{m-1}_h$ as a test function in~\eqref{eq:projection1} and using~\eqref{eq:projection2}, we see that 
				\begin{equation*}
					2(3u^{m+1}_h -4u^m_h +u^{m-1}_h,\tu^{m+1}_h-u^{m+1}_h) =0.
				\end{equation*}
				Thus, the previous identity simplifies to
				\begin{align*}
					2(3\tu^{m+1}_h -4u^m_h +u^{m-1}_h,\tu^{m+1}_h)& = 2(3u^{m+1}_h -4u^m_h +u^{m-1}_h,u^{m+1}_h) + 6(\tu^{m+1}_h-u^{m+1}_h,\tu^{m+1}_h)\\
					& = \norm{u^{m+1}_h}_{L^2}^2 - \norm{u^m_h}_{L^2}^2 +\norm{2u^{m+1}_h-u^m_h}_{L^2}^2-\norm{2 u^m_h - u^{m-1}_h}_{L^2}^2\\
					& +\norm{u^{m+1}_h-2u^m_h+ u^{m-1}_h}_{L^2}^2+ 3\norm{\tu^{m+1}_h}_{L^2}^2 - 3\norm{u^{m+1}_h}_{L^2}^2 + 3\norm{\tu^{m+1}_h-u^{m+1}_h}_{L^2}^2.
				\end{align*}
				Taking $v=4 \tu^{m+1}_h$ as a test function in~\eqref{eq:step1fully} and using this identity, as well as the skew-symmetry of $b$, we obtain
				\begin{multline}\label{eq:energytempx}
					\frac{1}{\Delta t}\Big(\norm{u^{m+1}_h}_{L^2}^2 - \norm{u^m_h}_{L^2}^2 +\norm{2u^{m+1}_h-u^m_h}_{L^2}^2-\norm{2 u^m_h - u^{m-1}_h}_{L^2}^2\\
					 +\norm{u^{m+1}_h-2u^m_h+ u^{m-1}_h}_{L^2}^2+ 3\norm{\tu^{m+1}_h}_{L^2}^2 - 3\norm{u^{m+1}_h}_{L^2}^2 + 3\norm{\tu^{m+1}_h-u^{m+1}_h}_{L^2}^2\Big) \\
					+ 4(\Grad p_h^m,\tu^{m+1}_h) +4 \mu(\Grad \tu^{m+1}_h,\Grad \tu^{m+1}_h)=4(\PUh f^{m+1},\tu^{m+1}_h).
				\end{multline}
				Next, we use the second step of the scheme to find an estimate for the term $4(\Grad p^m_h,\tu^{m+1}_h)$. We take $v = 2(u^{m+1}_h+\tu^{m+1}_h)+\frac{4\Delta t}{3}(\Grad p^{m+1}_h+\Grad p^m_h)$ as a test function in~\eqref{eq:projection1} and simplify:
				\begin{align*}
					0& = \frac{3}{4\Delta t}\left(2(u^{m+1}_h-\tu^{m+1}_h)+\frac{4\Delta t}{3}\Grad(p^{m+1}_h-p^m_h),2(u^{m+1}_h+\tu^{m+1}_h)+\frac{4\Delta t}{3}\Grad(p^{m+1}_h+p^m_h)\right)\\
					&=\frac{3}{4\Delta t}\norm{2u^{m+1}_h+\frac{4\Delta t}{3}\Grad p^{m+1}_h}_{L^2}^2 - \frac{3}{4\Delta t}\norm{2 \tu^{m+1}_h + \frac{4\Delta t}{3}\Grad p^{m}_h}_{L^2}^2\\
					&= \frac{3}{\Delta t}\norm{u^{m+1}_h}_{L^2}^2+\frac{4\Delta t}{3}\norm{\Grad p^{m+1}_h}_{L^2}^2 -  \frac{3}{\Delta t}\norm{\tu^{m+1}_h}_{L^2}^2-\frac{4\Delta t}{3}\norm{\Grad p^{m}_h}_{L^2}^2 - 4(\tu^{m+1}_h,\Grad p^m)
				\end{align*}
				where we also used~\eqref{eq:projection2} in the last line. Plugging this identity into~\eqref{eq:energytempx}, we obtain
					\begin{multline}\label{eq:energytempx2}
					\frac{1}{\Delta t}\Big(\norm{u^{m+1}_h}_{L^2}^2 - \norm{u^m_h}_{L^2}^2 +\norm{2u^{m+1}_h-u^m_h}_{L^2}^2-\norm{2 u^m_h - u^{m-1}_h}_{L^2}^2\\
					+\norm{u^{m+1}_h-2u^m_h+ u^{m-1}_h}_{L^2}^2+ 3\norm{\tu^{m+1}_h-u^{m+1}_h}_{L^2}^2\Big) \\
					+ \frac{4\Delta t}{3}\left(\norm{\Grad p^{m+1}_h}_{L^2}^2 - \norm{\Grad p^m_h}_{L^2}^2\right)  + 4\mu(\Grad \tu^{m+1}_h,\Grad \tu^{m+1}_h)=4(  f^{m+1},\tu^{m+1}_h),
				\end{multline}
				where we have also used~\eqref{eq:L2projection} for the last term.
				Multiplying~\eqref{eq:energytempx2} by $\Delta t$ and summing over $m=1,\dots, M-1$, we obtain 
				\begin{multline}\label{eq:prelimenergyestimate}
					E^M_h  	+\sum_{m=1}^{M-1}\norm{u^{m+1}_h-2u^m_h+ u^{m-1}_h}_{L^2}^2+ 3\sum_{m=1}^{M-1}\norm{\tu^{m+1}_h-u^{m+1}_h}_{L^2}^2 
					+ 4 \mu\Delta t \sum_{m=1}^{M-1}\norm{\Grad \tu^{m+1}_h}_{L^2}^2\\
					= E_h^1 +4 \Delta t\sum_{m=1}^{M-1}( f^{m+1},\tu^{m+1}_h).
				\end{multline}
				In order to derive an energy bound from this, we need to estimate $E_h^1$ which is given by
				\begin{equation*}
					E_h^1 = \norm{u^1_h}_{L^2}^2 + \norm{2u^1_h -u^0_h}_{L^2}^2+\frac{4\Delta t^2}{3}\norm{\Grad p^1_h}_{L^2}^2 \leq 7 \norm{u^1_h}_{L^2}^2 + 3\norm{u^0_h}_{L^2}^2 +\frac{4\Delta t^2}{3}\norm{\Grad p^1_h}_{L^2}^2,
				\end{equation*}
				thus we need an estimate on the $L^2$-norms of $u^1_h, u^0_h$ and $\Delta t \Grad p^1_h$ which can be obtained from analyzing the first and the zero'th time step, that were computed by using a backward Euler method, cf. Remark~\ref{rem:1ststep}, and by projection~\eqref{eq:zerothstep}. We start with the zero'th step: Taking $v=\Delta t (u^0_h+\Delta t \Grad p_h^0+\tu_h^0)$ as a test function in the first equation in~\eqref{eq:zerothstep} and using that $u^0_h$ satisfies the weak divergence constraint, we find
				\begin{equation*}
					0=(u^0_h-\tu^0_h+\Delta t \Grad p_h^0,u^0_h+\tu^0_h +\Delta t \Grad p_h^0)  =\norm{u_h^0}_{L^2}+\Delta t^2 \norm{\Grad p_h^0}_{L^2}^2 -\norm{\tu_h^0}_{L^2}^2
				\end{equation*}
				which implies that
				\begin{equation}\label{eq:u0estimate}
					\norm{u^0_h}_{L^2}^2+\Delta t^2\norm{\Grad p_h^0}_{L^2}^2=\norm{\tu^0_h}_{L^2}^2 = \norm{\PUh u_0}_{L^2}^2 \leq \norm{u_0}_{L^2}^2 \leq C.
				\end{equation}
				Next, we obtain a uniform in $h$ and $\Delta t$ bound on $u^1_h$ and $\Delta t \Grad p^1_h$ by using the first step~\eqref{eq:step1fully1step}--\eqref{eq:projection2step1}.
		This is in fact done in a very similar way. Taking $v=2\tu^1_h$ as a test function in~\eqref{eq:step1fully1step}, we have
				\begin{equation}\label{eq:firststepenergy}
					\frac{1}{\Delta t}\left(\norm{\tu^1_h}_{L^2}^2 + \norm{\tu^1_h-u^0_h}_{L^2}^2 - \norm{u^0_h}_{L^2}^2\right) - 2(p^0_h,\Div \tu^1_h)+2\mu \norm{\Grad \tu^1_h}_{L^2}^2 = 2(f^1,\tu^1_h).
				\end{equation}
				Then, taking $v=\frac{u^1_h+\tu^1_h}{\Delta t}+ \Grad(p^1_h + p^0_h)$ as a test function in~\eqref{eq:projection1step1} and using that $u^1_h$ is weakly divergence free, i.e., satisfies~\eqref{eq:projection2step1}, we have
				\begin{equation*}
					0 = \frac{1}{\Delta t^2}\left(\norm{u^1_h}_{L^2}^2 - \norm{\tu^1_h}_{L^2}^2\right)+\norm{\Grad p^1_h}_{L^2}^2 -\norm{\Grad p^0_h}_{L^2}^2 -\frac{2}{\Delta t}(\tu^1_h,\Grad p^0_h).
				\end{equation*}
				Plugging this into~\eqref{eq:firststepenergy} (after multiplying by $\Delta t$), we obtain
				\begin{equation}\label{eq:firststepenergy2}
					\frac{1}{\Delta t}\left(\norm{u^1_h}_{L^2}^2 + \norm{\tu^1_h-u^0_h}_{L^2}^2 - \norm{u^0_h}_{L^2}^2\right) +\Delta t\left(\norm{\Grad p^1_h}_{L^2}^2 -\norm{\Grad p^0_h}_{L^2}^2\right) +2\mu \norm{\Grad \tu^1_h}_{L^2}^2  %\\
					= 2(f^1,\tu^1_h).
				\end{equation}
Adding this $7\Delta t$ times to~\eqref{eq:prelimenergyestimate}, we obtain
	\begin{multline}\label{eq:prelimenergyestimate2}
	E^M_h  	+\sum_{m=1}^{M-1}\norm{u^{m+1}_h-2u^m_h+ u^{m-1}_h}_{L^2}^2+ 3\sum_{m=1}^{M-1}\norm{\tu^{m+1}_h-u^{m+1}_h}_{L^2}^2 
	+ 4 \mu\Delta t \sum_{m=0}^{M-1}\norm{\Grad \tu^{m+1}_h}_{L^2}^2+ 7\norm{\tu^1_h-u^0_h}_{L^2}^2\\
	\leq 10\norm{u_h^0}_{L^2}^2+ 7\Delta t^2\norm{\Grad p^0_h}_{L^2}^2 +4 \Delta t\sum_{m=0}^{M-1}( f^{m+1},\tu^{m+1}_h) + 3\Delta t(f^1,\tu_h^1)\\
	\leq 10\norm{u_0}_{L^2}^2 +4 \Delta t\sum_{m=0}^{M-1}( f^{m+1},\tu^{m+1}_h) + 3\Delta t(f^1,\tu_h^1).
\end{multline}		
				It remains to estimate the source term involving $f$ which will be done using a discrete version of Gr\"onwall's inequality, Lemma~\ref{lem:discretegronwall}.
						We first estimate with Young's inequality
				\begin{equation*}
					\left|(f^{m+1},\tu_h^{m+1})\right|\leq \frac{1}{2}\left(\norm{f^{m+1}}_{L^2}^2+\norm{\tu_h^{m+1}}_{L^2}^2\right),
				\end{equation*}
				and then use $v=u^{m+1}_h$ as a test function in~\eqref{eq:projection1} which yields with the weak divergence constraint~\eqref{eq:projection2},
				\begin{equation}\label{eq:utildebound}
					0= (u^{m+1}_h - \tu^{m+1}_h,u^{m+1}_h) = \frac{1}{2}\norm{u^{m+1}_h}_{L^2}^2 -\frac{1}{2}\norm{\tu^{m+1}_h}_{L^2}^2 + \frac12\norm{u^{m+1}_h - \tu^{m+1}_h}_{L^2}^2.
				\end{equation}
				Using this in the previous estimate, we obtain,
				\begin{equation*}
					\left|(f^{m+1},\tu_h^{m+1})\right|\leq \frac{1}{2}\left(\norm{f^{m+1}}_{L^2}^2+\norm{u_h^{m+1}}_{L^2}^2+ \norm{u^{m+1}_h-\tu^{m+1}_h}_{L^2}^2\right).
				\end{equation*}
					For the $m=0$ term, we estimate instead using the triangle inequality,
				\begin{equation*}
					\left|(f^{1},\tu_h^{1})\right|\leq \frac{1}{2}\left(\norm{f^{1}}_{L^2}^2+\norm{\tu_h^{1}}_{L^2}^2\right)\leq \frac{1}{2}\norm{f^{1}}_{L^2}^2+\norm{u_h^{0}}_{L^2}^2 +\norm{u^0_h-\tu_h^1}_{L^2}^2.
				\end{equation*}
				Thus, we can estimate in~\eqref{eq:prelimenergyestimate2}
				\begin{multline*}
					E^M_h  	+\sum_{m=1}^{M-1}\norm{u^{m+1}_h-2u^m_h+ u^{m-1}_h}_{L^2}^2+ 3\sum_{m=1}^{M-1}\norm{\tu^{m+1}_h-u^{m+1}_h}_{L^2}^2 
					+ 4 \mu\Delta t \sum_{m=0}^{M-1}\norm{\Grad \tu^{m+1}_h}_{L^2}^2+ 7\norm{\tu^1_h-u^0_h}_{L^2}^2\\
					\leq 10\norm{u_0}_{L^2}^2 +2 \Delta t\sum_{m=1}^{M-1}\left(\norm{f^{m+1}}^2_{L^2}+\norm{u^{m+1}_h}_{L^2}^2+\norm{u^{m+1}_h - \tu_h^{m+1}}_{L^2}^2\right) \\
					+ 7\Delta t\left(\frac12 \norm{f^1}_{L^2}^2 +\norm{u_0}_{L^2}^2 + \norm{u_h^0-\tu_h^1}_{L^2}^2\right),
				\end{multline*}		
				which we can recast to
					\begin{multline*}
					E^M_h  	+\sum_{m=1}^{M-1}\norm{u^{m+1}_h-2u^m_h+ u^{m-1}_h}_{L^2}^2+ (3-2\Delta t)\sum_{m=1}^{M-1}\norm{\tu^{m+1}_h-u^{m+1}_h}_{L^2}^2 
					\\
				+ 4 \mu\Delta t \sum_{m=0}^{M-1}\norm{\Grad \tu^{m+1}_h}_{L^2}^2	+ 7(1-\Delta t)\norm{\tu^1_h-u^0_h}_{L^2}^2\\
					\leq 10\norm{u_0}_{L^2}^2 +2 \Delta t\sum_{m=2}^{M}\left(\norm{f^{m}}^2_{L^2}+E_h^m\right)
					+ 7\Delta t\left(\frac12 \norm{f^1}_{L^2}^2 +\norm{u_0}_{L^2}^2 \right).
				\end{multline*}		
				We apply the discrete Gr\"onwall inequality, Lemma~\ref{lem:discretegronwall} to this identity to obtain with $\nu = 2$ and 
				\begin{equation*}
					a_m = E^m_h,\quad b_m = 10\norm{u_0}_{L^2}^2 +2 \Delta t\sum_{m=1}^{M}\norm{f^{m}}^2_{L^2}+  \Delta t\left( 1.5\norm{f^1}_{L^2}^2 +7\norm{u_0}_{L^2}^2 \right),
				\end{equation*}
				which is non-decreasing, to obtain
				\begin{multline*}
					E^M_h  	+\sum_{m=1}^{M-1}\norm{u^{m+1}_h-2u^m_h+ u^{m-1}_h}_{L^2}^2+ (3-2\Delta t)\sum_{m=1}^{M-1}\norm{\tu^{m+1}_h-u^{m+1}_h}_{L^2}^2 
					\\
					+ 4 \mu\Delta t \sum_{m=0}^{M-1}\norm{\Grad \tu^{m+1}_h}_{L^2}^2	+ 7(1-\Delta t)\norm{\tu^1_h-u^0_h}_{L^2}^2\\
					\leq \left(10\norm{u_0}_{L^2}^2 +2 \Delta t\sum_{m=1}^{M}\norm{f^{m}}^2_{L^2}+  \Delta t\left( 1.5\norm{f^1}_{L^2}^2 +7\norm{u_0}_{L^2}^2 \right)\right) (1-2\Delta t)^{-M}.
				\end{multline*}
				For $\Delta t>0$ small enough, we have $(1-2\Delta t)^{-1}\leq 1+3\Delta t$ and $2\Delta t\leq 1$ and  hence, we can bound the right hand side (using also $1+x\leq \exp(x)$) by
					\begin{multline*}
					E^M_h  	+\sum_{m=1}^{M-1}\norm{u^{m+1}_h-2u^m_h+ u^{m-1}_h}_{L^2}^2+ 2\sum_{m=1}^{M-1}\norm{\tu^{m+1}_h-u^{m+1}_h}_{L^2}^2 
					\\
					+ 4 \mu\Delta t \sum_{m=0}^{M-1}\norm{\Grad \tu^{m+1}_h}_{L^2}^2	+ 3\norm{\tu^1_h-u^0_h}_{L^2}^2\\
					\leq C\left( \norm{u_0}_{L^2}^2 +  \Delta t\sum_{m=1}^{M}\norm{f^{m}}^2_{L^2}     \right) e^{3M\Delta t},
				\end{multline*}
				which proves the result.
		\end{proof}

				\begin{remark}
					\label{rem:utildeL2bound}
					We also obtain from~\eqref{eq:utildebound}  that 
					\begin{equation}\label{eq:utildeL2bound}
					\norm{\tu^M_h}_{L^2}^2\leq C\left( \norm{u_0}_{L^2}^2 +  \Delta t\sum_{m=1}^{M}\norm{f^{m}}^2_{L^2}     \right) e^{3M\Delta t},
					\end{equation}
					for any $M\in \N$.
				\end{remark}
				
				\subsection{Solvability of the scheme}
				Next, we show that the scheme~\eqref{eq:step1fully}--\eqref{eq:projectionfullydiscrete} is well-posed, i.e., given $(u^m_h,u^{m-1}_h,\tu_h^m,\tu_h^{m-1},p_h^m)\in \Yh\times \Yh\times\Uh\times\Uh\times\Ph$, there is a unique $(\tu_h^{m+1},u^{m+1}_h,p_h^{m+1})\in \Uh\times\Yh\times\Ph$ solving~\eqref{eq:step1fully}--\eqref{eq:projectionfullydiscrete}.  The proof is mostly standard, an application of the Lax-Milgram theorem, and similar to~\cite[Lemma 3.4]{chorinprojection}. We include it here for completeness.
				\begin{lemma}
					\label{lem:solvability}
					For every $\Delta t, h>0$ and every $m=0,1,\dots$, given $(u^m_h,u^{m-1}_h,\tu_h^m,\tu_h^{m-1},p^m_h)\in \Yh\times\Yh\times\Uh\times\Uh\times\Ph$, then there exists a unique $(\tu_h^{m+1},u^{m+1}_h,p^{m+1}_h)\in \Uh\times\Yh\times\Ph$ solving~\eqref{eq:step1fully}--\eqref{eq:projectionfullydiscrete}.
				\end{lemma}
				\begin{proof}
					Without loss of generality, we let $m\geq 1$ and start by showing that if $(u_h^m,u^{m-1}_h,\tu_h^m,\tu_h^{m-1},p^m_h)\in\Yh\times\Yh\times\Uh\times\Uh\times\Ph$, then there exists a unique $\tu^{m+1}_h\in \Uh$ satisfying~\eqref{eq:step1fully}, i.e., the first step of the algorithm is well-posed (the case $m=0$ is proved with the same method but slightly different factors in front of the terms). To do so, we write~\eqref{eq:step1fully} as a linear variational problem: Find $\tu\in \Uh $ such that for all $v\in \Uh $,
						\begin{equation}
						\label{eq:elliptic}
						a^m( \tu ,v)= \mathcal{F}^m(v),
					\end{equation}
					where 
					\begin{equation}
						\label{eq:defa}
						a^m(\tu,v) =  	\frac32\left(\tu,v\right)+ \Delta t b(\hu^{m+1}_h,\tu ,v)+\Delta t\mu(\Grad \tu ,\Grad v) 
					\end{equation}
					and
					\begin{equation}
						\label{eq:deffn}
						\mathcal{F}^m(v) = 2(u^m_h,v)-\frac12(u_h^{m-1},v)-\Delta t (\Grad p_h^m, v)+\Delta t(\PUh f^{m+1},v).
					\end{equation}
					We show that $a^m$ is coercive and bounded on $\Uh$ and that $\mathcal{F}^m:\Uh\to \R$ is bounded. Then the Lax-Milgram theorem will imply well-posedness on $\Uh$ with the norm $\norm{v}_{\Uh}:=\norm{v}_{H^1_0}:= \norm{\Grad v}_{L^2}$.	To show the coercivity property, we plug in $v=\tu$ in the bilinear form $a^m$:
							\begin{equation*}
						a^m(\tu,\tu) =  \frac32	\norm{\tu}_{L^2}^2+ \Delta t b(\hu^{m+1}_h,\tu ,\tu)+\Delta t\mu(\Grad \tu ,\Grad \tu)
						= \frac32\norm{\tu}_{L^2}^2+ \Delta t\mu\norm{\Grad \tu}_{L^2}^2
						\geq c\norm{\tu}_{H^1_0}^2,
					\end{equation*}
				where $c>0$,	using the skew-symmetry of $b$. 
					This proves the coercivity on $\Uh$. Regarding the boundedness, we have
					\begin{align*}
						\left| 	a^m(\tu,v)\right| \leq  &  	\frac32\norm{\tu}_{L^2}\norm{v}_{L^2}+ \Delta t \norm{\hu^{m+1}_h}_{L^4}\norm{\tu}_{H^1_0}\norm{v}_{L^4}+\Delta t\mu\norm{\Grad \tu}_{L^2}\norm{\Grad v}_{L^2}\\
						\leq & \frac32\norm{\tu}_{L^2}\norm{v}_{L^2}+ \Delta t\left( \norm{\hu^{m+1}_h}_{L^4}+\mu\right)\norm{\tu}_{H^1_0}\norm{v}_{H^1_0}\leq C \norm{\tu}_{H^1_0}\norm{v}_{H^1_0}, 
					\end{align*}
					with an application of the Sobolev embedding theorem and Poincar\'e's inequality. Here we also used that $\tu^m_h,\tu_h^{m-1}\in H^1_0(\dom)$ by the discrete energy estimate~\eqref{eq:energybalance} and induction, and therefore $\hu_h^{m+1}\in H^1_0(\dom)$.
						Hence $a^m$ is bounded on $\Uh$. 
					Next, we prove that  $\mathcal{F}^m$ is bounded for every $m$:
						\begin{equation*}
				\begin{split}
								|\mathcal{F}^m(v)|&\leq 2\norm{u^m_h}_{L^2}\norm{v}_{L^2}+\frac12\norm{u^{m-1}_h}_{L^2}\norm{v}_{L^2}+\Delta t\norm{\Grad p_h^m}_{L^2}\norm{v}_{L^2}+\Delta t \norm{\PUh f^{m+1}}_{L^2}\norm{v}_{L^2}\\
								&\leq 2\norm{u^m_h}_{L^2}\norm{v}_{L^2}+\frac12\norm{u^{m-1}_h}_{L^2}\norm{v}_{L^2}+\Delta t\norm{\Grad p_h^m}_{L^2}\norm{v}_{L^2}+\Delta t \norm{f^{m+1}}_{L^2}\norm{v}_{L^2}\\
								& \leq C \norm{v}_{L^2}.
				\end{split}
						\end{equation*}
						Here we used that by the energy estimate~\eqref{eq:energybalance}, $u^m_h,u^{m-1}_h\in L^2(\dom)$ and $\Delta t \Grad  p^m_h\in L^2(\dom)$, and by assumption $f\in L^2([0,T]\times\dom)$.
						Thus, with the Poincar\'e inequality, we obtain boundedness of $\mathcal{F}^m$ on $\Uh$.
						We also note that $\Uh$ is dense in $H^1_0(\dom)$ and that these estimates hold for every fixed $\Delta  t, h>0$.
						This implies the solvability of Step 1 of the algorithm.
						For Step 2, we proceed similarly. We use the formulation of~\eqref{eq:projection1}--\eqref{eq:projection2} as a Poisson problem,~\eqref{eq:poisson}. So we seek $p\in \Ph$ such that for any $q\in \Ph$,
						\begin{equation*}
							\widetilde{a}(p,q) = \widetilde{\mathcal{F}}^m(q),
						\end{equation*}
						where $\widetilde{a}:\Ph\times\Ph\to \R$ is given by
						\begin{equation*}
							\widetilde{a}(p,q) = (\Grad p,\Grad q),
						\end{equation*}
						and $\widetilde{\mathcal{F}}^m:\Ph\to \R$ is given by
						\begin{equation*}
							\widetilde{\mathcal{F}}^m(q) = (\Grad p^m_h,\Grad q)-\left(\frac{3\Div \tu^{m+1}_h}{2\Delta t},q\right).
						\end{equation*}
						Clearly, $\widetilde{a}$ is coercive and bounded on $H^1(\dom)\cap L^2_0(\dom)$ and therefore on $\Ph$, and $\widetilde{\mathcal{F}}^m$ is bounded from $\Ph$ to $\R$ since $\tu^{m+1}_h \in H^1_0(\dom)$ and $\Grad p^m_h\in L^2(\dom)$ by the energy estimate and the solvability of the prediction step. Thus, applying the Lax-Milgran theorem again, $p^{m+1}_h$ exists and is unique, and therefore $u^{m+1}_h$ defined through~\eqref{eq:projection1} exists and is unique too. Moreover, the fact that $p^{m+1}_h$ satisfies~\eqref{eq:poisson} for all $q\in \Ph$ implies that $u^{m+1}_h$ satisfies the weak divergence constraint~\eqref{eq:projection2}.
					\end{proof}
					\section{Convergence analysis of the scheme}\label{sec:convergence}
				Using the a priori energy estimate derived in the previous section, we will now show that interpolations of the approximations defined by the scheme~\eqref{eq:step1fully}--\eqref{eq:projectionfullydiscrete} converge up to a subsequence to a weak solution of~\eqref{eq:NS}. For this purpose, we define the following interpolations in time:
					\begin{alignat}{2}
						%\begin{split}
						u_h(t)& = u_h^{m+1},\qquad &t\in (t^m,t^{m+1}],\label{eq:defuh}\\
						\bar{u}_h(t) &=2 u_h^m-u^{m-1}_h,\qquad &t\in (t^m,t^{m+1}],\\
						\widetilde{u}_h(t) &= \widetilde{u}_h^{m+1},\qquad &t\in (t^m,t^{m+1}],\\
							\hu_h(t) &= 2\tu_h^{m}-\tu^{m-1}_h,\qquad &t\in (t^m,t^{m+1}],\\
						p_h(t)& = p_h^{m+1},\qquad & t\in (t^m,t^{m+1}],\\
						f_h(t)& = \frac{1}{\Delta t}\int_{t^{m+1/2}}^{t^{m+3/2}}\PUh f(s)ds,\qquad &t\in (t^m,t^{m+1}],\label{eq:defph}
						%\end{split}
					\end{alignat}
					for $m=1,2,\dots$ and for $m=0$:
					\begin{align*}
						u_h(t) = \bar{u}_h(t) = u^1_h,\qquad &t\in (t^0,t^{1}],\\
						\tu_h(t) = \hu_h(t) = \tu^1_h,\qquad &t\in (t^0,t^{1}],\\
						p_h(t) = p_h^1,		\qquad &t\in (t^0,t^{1}],\\
							f_h(t) = \frac{1}{\Delta t}\int_{t^{1/2}}^{t^{3/2}}\PUh f(s)ds,\qquad &t\in [t^0,t^{1}],
					\end{align*}
					and
					\begin{equation*}
						\tu_h(s)=\hu_h(s)=\tu^0_h,\quad u_h(s)=\bar{u}_h(s)=u_h^0,\quad \text{and}\quad p_h(s)=p_h^0,\quad \text{for }\, s\leq 0.
					\end{equation*}
					From the discrete energy estimate in the last section, Lemma~\ref{lem:discenergy}  (and the triangle inequality), we obtain the following uniform a priori estimates for the sequences $\{u_h\}_{h>0}$, $\{\bar{u}_h\}_{h>0}$, $\{\tu_h\}_{h>0}$ and $\{\hu_h\}_{h>0}$:
					\begin{align*}
						&\{u_h\}_{h>0},\{\bar{u}_h\}_{h>0}\subset L^\infty(0,T;L^2(\dom)),\\
						&\{\tu_h\}_{h>0}, \{\hu_h\}_{h>0}\subset L^\infty(0,T;L^2(\dom))\cap L^2(0,T;H^1_0(\dom)).
					\end{align*}
					These uniform estimates imply, by the Banach-Alaoglu theorem, that there exist weakly convergent subsequences, which, for the ease of notation, we still denote by $h\to 0$,
					\begin{align}
						u_h\weakstar u,\quad 	\bar{u}_h\weakstar \bar{u},\quad	\widetilde{u}_h\weakstar \widetilde{u},\quad \hu_h\weakstar \hu, &\quad \text{in }\, L^\infty(0,T;L^2(\dom)),\label{eq:uhweakconv}\\
						\hu_h\weak \hu,\quad 	\widetilde{u}_h\weak \widetilde{u},&\quad \text{in }\, L^2(0,T;H^1_0(\dom)).\label{eq:Hhweakconv}
					\end{align}
					In order to prove convergence of the scheme, we will need to derive strong precompactness of these sequences in $L^2([0,T]\times\dom)$. We will do this by using Simon's compactness criterion~\cite{Simon1987}, stated as Theorem~\ref{thm:Simonlemma} in the Appendix, for the sequence $\{\tu_h\}_{h>0}$. The main issue here is that while the first condition in Theorem~\ref{thm:Simonlemma} is readily availabe thanks to the energy estimate (with $B=L^2(\dom)$), the second criterion requires more work and combining estimates for the different approximations. As a first step, we show that the sequences $\{\tu_h\}_{h>0}$, $\{u_h\}_{h>0}$, $\{\bar{u}_h\}_{h>0}$ and $\{\hu_h\}_{h>0}$ have the same limits:
					\begin{lemma}
						\label{lem:samelimits}
						Assume that  $\Delta t = o_{h\to 0}(1)$. Then, we have 
						$u=\bar{u}=\tu=\hu$ a.e. in $[0,T]\times\dom$ and in $L^2([0,T]\times\dom)$, and
						\begin{align}
							\label{eq:uubarconv}
							\lim_{h\to 0}\norm{u_h-\bar{u}_h}_{L^2([0,T]\times\dom)} &= 0,\\
							\label{eq:utildeuconv}
							\lim_{h\to 0}\norm{u_h-\tu_h}_{L^2([0,T]\times\dom)}& = 0,\\
								\label{eq:ubaruhatconv}
							\lim_{h\to 0}\norm{\bar{u}_h-\hu_h}_{L^2([0,T]\times\dom)}& = 0.
						\end{align}	
					\end{lemma}
					\begin{proof}
						We have
						\begin{align*}
							\norm{{u}_h-\tu_h}_{L^2([0,T]\times\dom)}^2 & = \Delta t\sum_{m=1}^{N-1}\int_{\dom}\left|u^{m+1}_h-\tu^{m+1}_h\right|^2 dx +\Delta t \int_{\dom}\left|u^{1}_h-\tu^{1}_h\right|^2 dx\\
							&\leq C\Delta t (\norm{u_0}_{L^2}^2+\norm{f}_{L^2([0,T]\times\dom)}^2)+ 2\Delta t\left(\norm{u_h^1}_{L^2}^2+\norm{\tu_h^1}_{L^2}^2\right)\\
							& \leq C\Delta t (\norm{u_0}_{L^2}^2+\norm{f}_{L^2([0,T]\times\dom)}^2)
						\end{align*}
						by the discrete energy estimate, Lemma~\ref{lem:discenergy}, and where we also used that $\norm{\tu_h^1}_{L^2}^2\leq 2\norm{u_h^0}_{L^2}^2 + 2\norm{u_h^0-\tu_h^1}^2_{L^2}$. This shows~\eqref{eq:utildeuconv}. 
						Using this, it then follows easily that for any test function $\varphi\in L^2([0,T]\times\dom)$, since $u_h,\tu_h$ converge weakly in $L^2([0,T]\times\dom)$,
						\begin{multline}\label{eq:samelim3}
							\int_0^T\!\!\int_{\dom}{u}\cdot \varphi\, dx dt = \lim_{h\to 0} \int_0^T\!\!\int_{\dom}{u}_h\cdot \varphi\, dx dt\\
							= \lim_{h\to 0}\int_0^T\!\!\int_{\dom}({u}_h-\tu_h)\cdot \varphi\, dx dt+ \lim_{h\to 0}\int_0^T\!\!\int_{\dom}\tu_h\cdot \varphi \,dx dt =\int_0^T\!\!\int_{\dom}\tu\cdot \varphi\, dx dt,
						\end{multline}
						hence $\tu={u}$ a.e. in $[0,T]\times\dom$ and in $L^2([0,T]\times\dom)$.
						In the same way, we compute
						\begin{align*}
							\norm{\bar{u}_h-u_h}_{L^2([0,T]\times\dom)}^2 & = \Delta t\sum_{m=1}^{N-1}\int_{\dom}\left|u^{m+1}_h-2u^m_h+u^{m-1}_h\right|^2 dx \\
							&\leq C\Delta t (\norm{u_0}_{L^2}^2+\norm{f}_{L^2([0,T]\times\dom)}^2).
						\end{align*}
						Thus as $\Delta t,h \to 0$, we obtain~\eqref{eq:uubarconv}, and in the same way as in~\eqref{eq:samelim3}, one shows that the limits $u$ and $\bar{u}$ agree almost everywhere.  Finally, to prove~\eqref{eq:ubaruhatconv}, we use the triangle inequality and the energy balance~\eqref{eq:energybalance}:
							\begin{align*}
							\norm{\bar{u}_h-\hu_h}_{L^2([0,T]\times\dom)}^2 & = \Delta t\sum_{m=1}^{N-1}\int_{\dom}\left|2u^m_h-u^{m-1}_h-2\tu_h^{m}+\tu_h^{m-1}\right|^2 dx +\Delta t \int_{\dom} |u_h^1 - \tu_h^1|^2 dx \\
							&\leq 9\Delta t\sum_{m=1}^{N-1}\int_{\dom}\left|u^m_h-\tu_h^{m}\right|^2 dx + 2 \Delta t\sum_{m=1}^{N-1}\int_{\dom}\left|u^{m-1}_h-\tu_h^{m-1}\right|^2 dx\\
							&\leq C\Delta t (\norm{u_0}_{L^2}^2+\norm{f}_{L^2([0,T]\times\dom)}^2).
						\end{align*}
						This implies, as in~\eqref{eq:samelim3}, that the limits $\bar{u}$ and $\hu$ agree, and thus all the limits agree.
					\end{proof}
				
				Next, we note that the limit $u$ is weakly divergence free:
				\begin{lemma}\label{lem:divfree2}
					Any weak limit $u$ of the sequence $\{{u}_h\}_{h>0}$ satisfies for a.e. $t\in [0,T]$,
					\begin{equation*}
						(u,\Grad q) = 0,\quad \forall \, q\in H^1(\dom).
					\end{equation*}
					
				\end{lemma}
				\begin{proof}
					This was proved in~\cite[Lemma 5.6]{chorinprojection}, but we will repeat the proof here for convenience.
					As $u$ is the weak $L^2$-limit of $\{{u}_h\}_{h>0}$ , we have for almost every $t\in [0,T]$ and any $q\in H^1(\dom)$,
					\begin{equation*}
						(u,\Grad q) = \lim_{h\to 0} ({u}_h,\Grad q).
					\end{equation*}
					Therefore we can write
					\begin{equation*}
						|({u}_h ,\Grad q)| = 	|({u}_h ,\Grad (q-\PPh q)) |\leq \norm{{u}_h}_{L^2}\norm{\Grad (q-\PPh q)}_{L^2}\stackrel{h\to 0}{\longrightarrow} 0,
					\end{equation*}
					since ${u}_h\in L^\infty([0,T];L^2(\dom))$ and $\PPh q \to q$ in $H^1(\dom)$ as $h\to 0$ for any $q\in H^1(\dom)$ by~\eqref{eq:H1}. Thus, in the limit $(u,\Grad q) = 0$ for any $q\in H^1(\dom)$  and a.e. $t\in [0,T]$.
				\end{proof}
				Next, we derive a preliminary technical result that will  later be improved to an estimate on the uniform time continuity of $\tu_h$ to obtain the second criterion in Theorem~\ref{thm:Simonlemma}.
				To do so, we adapt a result from~\cite[Lemma 3.3]{Eymard2024} and~\cite[Lemma 5.5]{chorinprojection}:
				\begin{lemma}
					\label{lem:timecontinuityutilde}
					Denote $Z = (H^1_{\Div}(\dom)\cap H^s(\dom))^*$. Assume that $h^s\leq C {\Delta t}$ where $2\leq s\leq k+1$ and $k$ is the polynomial degree of $\Uh$. Then the aproximations $\tu_h$ computed by~\eqref{eq:step1fully}--\eqref{eq:projectionfullydiscrete} satisfy for any $0\leq \tau <T$,  
				\begin{equation}
						\label{eq:timecontutildeweak}
					\int_{\Delta t}^{T-\tau} \norm{3(\tu_h(t+\tau)-\tu_h(t))-(\tu_h(t+\tau-\Delta t)-\tu_h(t-\Delta t))}_{Z}^2 dt \leq C \tau (\tau +\Delta t).
					\end{equation}
				\end{lemma}
				\begin{proof}
					Let $\tau>0$ and $t\in (\Delta t,T-\tau]$. Assume that $\tau>\Delta t$, since otherwise there is nothing to prove. 
					We let $m_1\in \N$ such that $\tu_h(t) = \tu_h^{m_1+1}$ and $m_2\in \N$ such that $\tu_h(t+\tau) = \tu_h^{m_2+1}$.  Then
					\begin{align*}
						3(\tu_h(t+\tau)-\tu_h(t))-(\tu_h(t+\tau-\Delta t)-\tu_h(t-\Delta t)) &= 3 \tu_h^{m_2+1}-\tu_h^{m_2}-3 \tu_h^{m_1+1}+ \tu_h^{m_1} \\
						&= \sum_{m=m_1+1}^{m_2}(3 \tu^{m+1}_h-4\tu_h^m+ \tu_h^{m-1}).
					\end{align*}
					Thus, we can take a test function $v\in H^1_{\Div}\cap H^s(\dom)$ and plug in the scheme~\eqref{eq:step1fully}--\eqref{eq:projectionfullydiscrete},
					\begin{align*}
						&(3(\tu_h(t+\tau)-\tu_h(t))-(\tu_h(t+\tau-\Delta t)-\tu_h(t-\Delta t)) ,v) \\
						&\qquad  =  \sum_{m=m_1+1}^{m_2}(3 \tu^{m+1}_h-4\tu_h^m+ \tu_h^{m-1},v)\\
						&\qquad  =  \sum_{m=m_1+1}^{m_2}(3 \tu^{m+1}_h-4\tu_h^m+ \tu_h^{m-1},\PUh v)\\
						&\qquad  = -2\Delta t   \sum_{m=m_1+1}^{m_2}\underbrace{b(\hu^{m+1},\tu^{m+1},\PUh v)}_{\text{I}} - 2\Delta t\sum_{m=m_1+1}^{m_2}\underbrace{\mu(\Grad \tu^{m+1},\Grad \PUh v)}_{\text{II}}  \\
						&\qquad  \quad -\Delta t\sum_{m=m_1+1}^{m_2} \underbrace{\left(\frac{14}{3}\Grad p^{m}_h-\frac{10}{3}\Grad p^{m-1}_h+\frac23\Grad p^{m-2},\PUh v\right)}_{\text{III}}+2\Delta t \sum_{m=m_1+1}^{m_2}\underbrace{(f^{m+1},\PUh v)}_{\text{IV}}
					\end{align*}
						We estimate each of the terms I -- IV:
					\begin{align*}
						|\text{I}|&=\left|b(\hu_h^{m+1},\tu_h^{m+1},\PUh v)\right|\\
						& \leq \left(\norm{\hu^{m+1}_h}_{L^2} \norm{\Grad\tu_h^{m+1}}_{L^2}+\frac12\norm{\Div \hu_h^{m+1}}_{L^2}\norm{\tu_h^{m+1}}_{L^2} \right)\norm{\PUh v}_{L^\infty}\\
						& \leq C\left(\norm{\hu^{m+1}_h}_{L^2} \norm{\Grad\tu_h^{m+1}}_{L^2} +\norm{\Grad\hu_h^{m+1}}_{L^2}\norm{\tu_h^{m+1}}_{L^2}\right)\norm{ v}_{L^\infty}\\
						& \leq C\left(\norm{\hu^{m+1}_h}_{L^2} \norm{\Grad\tu_h^{m+1}}_{L^2} +\norm{\Grad\hu_h^{m+1}}_{L^2}\norm{\tu_h^{m+1}}_{L^2}\right)\norm{ v}_{H^2},
					\end{align*}
					using~\eqref{eq:Lr}.
					For the second term, we have
					\begin{align*}
						|\text{II}|
						& \leq\mu \norm{\Grad \tu^{m+1}_h}_{L^2(\dom)}\norm{\Grad\PUh v}_{L^2(\dom)}\\
						& \leq C \norm{\Grad \tu^{m+1}_h}_{L^2(\dom)}\norm{ v}_{H^1(\dom)},
					\end{align*}
					using~\eqref{eq:H1}.
					For the third term, III, we have, using that $v$ is divergence free,  the properties of the $L^2$-projection,~\eqref{eq:L2projproperties}, and that by the energy estimate $\norm{\Grad p}_{L^\infty(0,T;L^2(\dom))}\leq C \Delta t^{-1}$,
					\begin{align}\label{eq:pressuretimecont}
						|\text{III}|&\leq C\left(\frac{14}{3}\norm{\Grad p^m_h}_{L^2(\dom)}+\frac{10}{3}\norm{\Grad p^{m-1}_h}_{L^2(\dom)}+\frac23\norm{\Grad p^{m-2}}_{L^2(\dom)}\right)\norm{\PUh v-v}_{L^2( \dom)}\\
						& \leq C\left(\frac{14}{3}\norm{\Grad p^m_h}_{L^2(\dom)}+\frac{10}{3}\norm{\Grad p^{m-1}_h}_{L^2(\dom)}+\frac23\norm{\Grad p^{m-2}}_{L^2(\dom)}\right)h^s\norm{ v}_{H^s( \dom)}\notag\\
						& \leq C \frac{h^s}{\Delta t}\norm{ v}_{H^s( \dom)}\notag\\
						& \leq  C \norm{ v}_{H^2( \dom)},\notag
					\end{align}
					under the condition that $h^s\leq C {\Delta t}$, and where we used~\eqref{eq:L2approx}.
					Finally, for the last term,  IV, we have
					\begin{equation*}
						|\text{IV}|  \leq \norm{f^{m+1}}_{ L^2(\dom)}\norm{v}_{ L^2(\dom)} .
					\end{equation*}
					Combining the estimates I -- IV, we have, after taking the supremum over $v\in H^1_{\Div}(\dom)\cap H^s(\dom)$,
					\begin{align*}
						&\norm{3(\tu_h(t+\tau)-\tu_h(t))-(\tu_h(t+\tau-\Delta t)-\tu_h(t-\Delta t)}_Z\\
						&\quad  \leq C \Delta t  \sum_{m=m_1+1}^{m_2} \Big( \norm{\hu^{m+1}_h}_{L^2} \norm{\Grad\tu_h^{m+1}}_{L^2}+ \norm{\Grad\hu^{m+1}_h}_{L^2} \norm{\tu_h^{m+1}}_{L^2}+ \norm{\Grad \tu^{m+1}_h}_{L^2(\dom)}\\
						&\qquad \hphantom{\leq C \Delta t  \sum_{m=m_1+1}^{m_2} \Big(}+ 1+ \norm{f^{m+1}}_{ L^2(\dom)} \Big).
					\end{align*}
					We square this identity and use that $\Delta t (m_2-m_1) \leq \tau +\Delta t$,
					\begin{align*}
						&\norm{3(\tu_h(t+\tau)-\tu_h(t))-(\tu_h(t+\tau-\Delta t)-\tu_h(t-\Delta t)}_Z^2 \\
						&\quad \leq C (\tau+\Delta t)\Delta t  \sum_{m=m_1+1}^{m_2} \Big( \norm{\hu^{m+1}_h}_{L^2}^2 \norm{\Grad\tu_h^{m+1}}^2_{L^2}+\norm{\Grad \hu_h^{m+1}}_{L^2}\norm{\tu_h^{m+1}}_{L^2}\\
						&\qquad \hphantom{ \leq C (\tau+\Delta t)\Delta t  \sum_{m=m_1+1}^{m_2} \Big( }+ \norm{\Grad \tu^{m+1}_h}^2_{L^2(\dom)}+ 1+ \norm{f^{m+1}}^2_{ L^2(\dom)} \Big).
					\end{align*}
				Then we integrate this identity over time to obtain, using the energy estimate, Lemma~\ref{lem:discenergy},  
					\begin{equation*}
						\begin{split}
							&\int_{\Delta t}^{T-\tau}\norm{3(\tu_h(t+\tau)-\tu_h(t))-(\tu_h(t+\tau-\Delta t)-\tu_h(t-\Delta t)}_Z^2  dt  \\
							& \leq C (\tau+\Delta t)\int_{\Delta t}^{T-\tau} \int_t^{t+\tau} \Big( \norm{\hu _h}_{L^2}^2 \norm{\Grad\tu_h}^2_{L^2}+\norm{\tu _h}_{L^2}^2 \norm{\Grad\hu_h}^2_{L^2}\\
							&\quad \hphantom{\leq C (\tau+\Delta t)\int_0^{T-\tau} \int_t^{t+\tau} \Big( } + \norm{\Grad \tu_h}^2_{L^2(\dom)}+ 1+ \norm{f_h}^2_{ L^2(\dom)} \Big) ds dt\\
							& \leq C (\tau+\Delta t)\tau \int_{\Delta t}^{T-\tau}  \Big( \norm{\hu _h}_{L^2}^2 \norm{\Grad\tu_h}^2_{L^2}+\norm{\tu _h}_{L^2}^2 \norm{\Grad\hu_h}^2_{L^2}\\
							&\quad \hphantom{\leq C (\tau+\Delta t)\int_0^{T-\tau}  \Big( } + \norm{\Grad \tu_h}^2_{L^2(\dom)}+ 1+ \norm{f_h}^2_{ L^2(\dom)} \Big) ds\\
							& \leq C (\tau+\Delta t)\tau, 
						\end{split}
					\end{equation*}
					which is what we wanted to prove.
				\end{proof}
				Next, we recall the following Lions-like lemma that was proved in~\cite[Lemma 4.5]{chorinprojection}:
				\begin{lemma}\label{lem:lionslike}
					Denote $Z = (H^k(\dom)\cap H^1_{\Div}(\dom))^*$ for some $k\in \N$. Then for any $\eta>0$ there exists a constant $C_\eta>0$ and $h_\eta>0 $ such that for all $h\leq h_{\eta}$ we have
					\begin{equation}\label{eq:lionslike}
						\norm{\PVh w}_{L^2(\dom)}\leq \eta \norm{w}_{H^1_0(\dom)} + C_\eta \norm{w}_Z\quad \forall w\in H^1_0(\dom).
					\end{equation}
				\end{lemma}	
				Now, we will combine this result with Lemma~\ref{lem:timecontinuityutilde} to show uniform in $h$ time continuity of $\tu_h$ in $L^2([0,T]\times\dom)$. The proof of the following lemma uses ideas from~\cite[Lemma 3.6]{Eymard2024} and~\cite[Lemma 5.6]{chorinprojection}, however, requires some additional work because we only have the estimate~\eqref{eq:timecontutildeweak} on averages of $\tu_h$ at different times. 
				\begin{lemma}\label{lem:timecont}
					The approximations $\tu_h$ computed by the scheme~\eqref{eq:step1fully}--\eqref{eq:projectionfullydiscrete} satisfy
					\begin{equation}
						\label{eq:timecont2}
						\int_0^{T-\tau}\norm{\tu_h(t+\tau)-\tu_h(t)}_{L^2(\dom)}^2 dt \longrightarrow 0,\quad \text{as }\, \tau\to 0,
					\end{equation}
					uniformly with respect to $h,\Delta t>0$.
				\end{lemma}
				\begin{proof}
					We recall from Remark~\ref{rem:PVhonUh} that ${u}_h$ is given by the projection of $\tu_h$ onto $\Vh$, specifically ${u}_h = \PVh \tu_h$. Then, we can estimate the quantity in~\eqref{eq:timecont2} as follows, using that $\PVh$ is an orthogonal projection:
					\begin{multline*}
						\int_0^{T-\tau}\norm{\tu_h(t+\tau)-\tu_h(t)}_{L^2(\dom)}^2 dt = 	\underbrace{\int_0^{T-\tau}\norm{\tu_h(t+\tau)-\tu_h(t) - ({u}_h(t+\tau)-{u}_h(t))}_{L^2(\dom)}^2 dt}_{(i)} \\
						+  \underbrace{\int_0^{T-\tau}\norm{ {u}_h(t+\tau)- {u}_h(t)}_{L^2(\dom)}^2 dt}_{(ii)}.
					\end{multline*}
					We let $\epsilon>0$ arbitrary and given. For the first term, since $\norm{\tu_h-{u}_h}_{L^2}^2\leq C \Delta t$  by Lemma~\ref{lem:samelimits}, we have that there is $h_1>0$ (and corresponding $\Delta t $) small enough such that for any $0<h<h_1$, we have
					\begin{equation*}
						(i) \leq 2\int_0^{T-\tau}\norm{\tu_h(t+\tau)  -  {u}_h(t+\tau) }_{L^2(\dom)}^2 dt + 2\int_0^{T}\norm{\tu_h(t)-{u}_h(t)}_{L^2(\dom)}^2 dt < \frac{\epsilon}{8}.
					\end{equation*}
					For the second term (ii), we combine Lemma~\ref{lem:lionslike}   with the weak time continuity of $\tu_h$, Lemma~\ref{lem:timecontinuityutilde}, the discrete energy inequality, Lemma~\ref{lem:discenergy}, and the fact that ${u}_h= \PVh \tu_h$ as follows:
					First, we note that by the triangle inequality,
						\begin{multline}\label{eq:anotherawkwardestimate}
						\int_{0}^{T-\tau}\norm{u_h(t+\tau)-u_h(t)}_{L^2}^2 dt\\
						 \leq 	\underbrace{\frac34\int_{\Delta t}^{T-\tau} \norm{3(u_h(t+\tau)-u_h(t))-(u_h(t+\tau-\Delta t)-u_h(t-\Delta t))}_{L^2}^2 dt}_{(a)} \\
						+ \underbrace{\frac34 \int_{\Delta t}^{T-\tau}\norm{u_h(t+\tau)-u_h(t+\tau -\Delta t)}^2_{L^2} dt}_{(b)} + \underbrace{\frac34 \int_{\Delta t}^{T-\tau}\norm{u_h(t)-u_h(t -\Delta t)}^2_{L^2} dt}_{(c)}\\
						+ \underbrace{\int_0^{\Delta t} \norm{u_h(t+\tau)-u_h(t)}_{L^2}^2 dt}_{(d)}.
					\end{multline} 
					For the first term, we have by Lemma~\ref{lem:lionslike} that for any $\eta>0$ there exists $C_\eta>0$ and $h_\eta>0$ such that for $0<h<\min\{h_\eta,h_1\}$,
					\begin{align*}
						|(a)| &\leq \eta \int_{\Delta t}^{T-\tau} \norm{3(\tu_h(t+\tau)-\tu_h(t))-(\tu_h(t+\tau-\Delta t)-\tu_h(t-\Delta t))}_{H^1_0}^2 dt \\
						&+ C_\eta \int_{\Delta t}^{T-\tau} \norm{3(\tu_h(t+\tau)-\tu_h(t))-(\tu_h(t+\tau-\Delta t)-\tu_h(t-\Delta t))}_{Z}^2 dt\\
						& \leq \eta C_{E_0,f}+ C_\eta \int_{\Delta t}^{T-\tau} \norm{3(\tu_h(t+\tau)-\tu_h(t))-(\tu_h(t+\tau-\Delta t)-\tu_h(t-\Delta t))}_{Z}^2 dt,
					\end{align*}
					where we used the discrete energy estimate, Lemma~\ref{lem:discenergy} for the last inequality and we denoted again $Z:= (H^1_{\Div}(\dom)\cap H^k(\dom))^*$. Now we first choose $\eta = \epsilon / (16C_{E_0,f})$ and then, since we know by Lemma~\ref{lem:timecontinuityutilde} that the second term goes to zero as $\tau\to 0$, we have that for any $0<\tau<\bar{\tau}$ small enough that 
					\begin{equation*}
						C_\eta \int_{\Delta t}^{T-\tau} \norm{3(\tu_h(t+\tau)-\tu_h(t))-(\tu_h(t+\tau-\Delta t)-\tu_h(t-\Delta t))}_{Z}^2 dt \leq \frac{\epsilon}{16}.
					\end{equation*}
					Thus $|(a)|\leq \frac{\epsilon}{8}$. Terms $(b)$ and $(c)$ are estimated in the same way, so we just consider $(b)$. We have by the triangle inequality and the energy inequality, Lemma~\ref{lem:discenergy},
					\begin{align*}
						(b) & = \frac34 \int_{\tau+\Delta t}^{T}\norm{u_h(t)-u_h(t-\Delta t)}^2_{L^2} dt\\
						& \leq \frac{3}{8}  \int_{\tau+\Delta t}^{T}\norm{3(u_h(t)-u_h(t-\Delta t))-(u_h(t-\Delta t)-u_h(t-2\Delta t))}^2_{L^2} dt\\
						&\quad + \frac{3}{8} \int_{\tau+\Delta t}^{T}\norm{u_h(t)-2u_h(t-\Delta t)+u_h(t-2\Delta t))}^2_{L^2} dt\\
						& \leq  \frac{3}{8}  \int_{\tau+\Delta t}^{T}\norm{3(u_h(t)-u_h(t-\Delta t))-(u_h(t-\Delta t)-u_h(t-2\Delta t))}^2_{L^2} dt + C_{E_0,f}\Delta t.
					\end{align*}
					We pick $0<h_2\leq \min\{h_{\eta},h_1\}$ (and corresponding $\Delta t$) such that $C_{E_0,f}\Delta t \leq \frac{\epsilon}{16}$. Then, we again use Lemma~\ref{lem:lionslike} followed by Lemma~\ref{lem:timecontinuityutilde} for $\tau = \Delta t$ to estimate
					\begin{align*}
						(b)& \leq \frac{3}{8}  \int_{\tau+\Delta t}^{T}\norm{3(u_h(t)-u_h(t-\Delta t))-(u_h(t-\Delta t)-u_h(t-2\Delta t))}^2_{L^2} dt + \frac{\epsilon}{16}\\
						& \leq \eta C_{E_0,f} + C_\eta \int_{\tau+\Delta t}^{T}\norm{3(\tu_h(t)-\tu_h(t-\Delta t))-(\tu_h(t-\Delta t)-\tu_h(t-2\Delta t))}^2_{Z} dt+ \frac{\epsilon}{16} \\
						& \leq \frac{\epsilon}{8}+ C_\eta \int_{\tau+\Delta t}^{T}\norm{3(\tu_h(t)-\tu_h(t-\Delta t))-(\tu_h(t-\Delta t)-\tu_h(t-2\Delta t))}^2_{Z} dt. 
					\end{align*}
					Now we can find $0<h_3\leq \min\{h_\eta,h_1,h_2\}$ such that for the $0<h<h_3$ the corresponding $\Delta t$ satisfy
					\begin{equation*}
						C_\eta \int_{\tau+\Delta t}^{T}\norm{3(\tu_h(t)-\tu_h(t-\Delta t))-(\tu_h(t-\Delta t)-\tu_h(t-2\Delta t))}^2_{Z} dt\leq \frac{\epsilon}{8}.
					\end{equation*}
					Thus $(b)\leq \frac{\epsilon}{4}$. Similarly, $(c)\leq \frac{\epsilon}{4}$. For term $(d)$, we use that $u_h\in L^\infty(0,T;L^2(\dom))$ and therefore,
					\begin{equation*}
						(d)\leq 4\Delta t  \norm{u_h}_{L^\infty(0,T;L^2(\dom))}.
					\end{equation*}
					Hence for $\Delta t,h$ sufficiently small, we have $(d)\leq \frac{\epsilon}{8}$.
					Thus, $(ii)\leq \frac{7\epsilon}{8}$ and hence 
					\begin{equation*}
							\int_0^{T-\tau}\norm{\tu_h(t+\tau)-\tu_h(t)}_{L^2(\dom)}^2 dt\leq \epsilon,
					\end{equation*}
					which proves the result since $\epsilon$ was arbitrary.
				\end{proof}
				The result of this lemma, combined with the uniform $L^2([0,T];H^1_0(\dom))$-bound on the $\tu_h$ coming from the energy estimate, Lemma~\ref{lem:discenergy}, allow us to apply Simon's version of the Aubin-Lions-Simon lemma, Theorem~\ref{thm:Simonlemma}, to conclude that a subsequence of $\{\tu_h\}_{h>0}$ converges strongly in $L^2([0,T]\times\dom)$ to the limit $u$. By Lemma~\ref{lem:samelimits}, this implies that also $\{{u}_h\}_{h>0}$, $\{\bar{u}_h\}_{h>0}$ and $\{\hu_h\}_{h>0}$ have strongly convergent subsequences to the same limit $u$. 
				
					Finally, we use these estimates to prove:
					\begin{theorem}\label{thm:convergence}
						The sequences $\{u_h,\bar{u}_h,\widetilde{u}_h,\hu_h \}_{h>0}$ converge up to a subsequence to a Leray-Hopf solution $u$ of~\eqref{eq:NS} as in Definition~\ref{def:weaksol} under the condition that {$h^{k+1} = o({\Delta t})$}, where $k$ is the polynomial degree of the finite element space $\Uh$, as $h,\Delta t\to 0$.
					\end{theorem}
					\begin{proof}
						Using the definitions of the interpolations~\eqref{eq:defuh}--\eqref{eq:defph}, we can rewrite the numerical scheme~\eqref{eq:step1fully}--\eqref{eq:projectionfullydiscrete} as
						\begin{equation}
							\label{eq:udiscinterp}
							\left(D_t^- u_h,v\right)+ b(\hu_h,\tu_h,v)+\mu(\Grad \tu_h,\Grad v) = (p_h,\Div v)+(f_h,v),
						\end{equation}
						and 
						\begin{equation}
							\label{eq:divconstraintinterp}
							(u_h,\Grad q)=0,
						\end{equation}
						with the test functions in the same space as in~\eqref{eq:step1fully}--\eqref{eq:projectionfullydiscrete}, and where we denoted
						\begin{equation*}
							D_t^- u_h(t) := \frac{3 u_h(t)-4 u_h(t-\Delta t)+ u_h(t-2\Delta t)}{2\Delta t},
						\end{equation*}
						for $t>\Delta t$ and 
						\begin{equation*}
							D_t^- u_h(t) : = \frac{u_h(t)-u_h(t-\Delta t)}{\Delta t},
						\end{equation*}
						for $t\in [0,\Delta t]$.
						By the previous considerations, weak convergences~\eqref{eq:uhweakconv}--\eqref{eq:Hhweakconv}, Lemma~\ref{lem:discenergy}, Lemma~\ref{lem:timecont}, Lemma~\ref{lem:samelimits}, and Theorem~\ref{thm:Simonlemma}, we can improve~\eqref{eq:uhweakconv}--\eqref{eq:Hhweakconv} to
						\begin{align}
							u_h,\bar{u}_h,\hu_h,\widetilde{u}_h\to u,&\quad \text{in }\, L^2(0,T;L^2(\dom)),\label{eq:uhstrongconv}\\
							\widetilde{u}_h,\hu_h\weak u,&\quad \text{in }\, L^2(0,T;H^1_0(\dom)).\label{eq:tuhweakconvagain}
						\end{align}
						From Lemma~\ref{lem:divfree2}, we also obtain that the limit $u$ is weakly divergence free. Since $u\in L^2([0,T];H^1_0(\dom))$, this implies that $\Div u=0$ a.e. in $[0,T]\times\dom$.
						
						Now we take a test function $v \in C^\infty_c((0,T)\times \dom)$ that is divergence free and choose $\Delta t$ small enough such that $v(t,\cdot) = 0$ on $[T-2\Delta t, T]$ and on $[0,2\Delta t]$. We integrate in time and rewrite~\eqref{eq:udiscinterp} as 
							\begin{align}
								\label{eq:uapp}
								&\int_0^T\Big[\underbrace{\left(D_t^- u_h,v\right)}_{(i)}+ \underbrace{b(\hu_h,\tu_h,v)}_{(ii)}+\underbrace{\mu(\Grad \tu_h,\Grad v)}_{(iii)} -\underbrace{(f_h,v)}_{(iv)}\Big] dt\\
								&=\int_0^T\Big[\underbrace{\left(D^-_t u_h,v-\PUh v\right)}_{(v)}+ \underbrace{b(\hu_h,\tu_h,v-\PUh v)}_{(vi)}+\underbrace{\mu(\Grad \tu_h,\Grad (v-\PUh v))}_{(vii)}\Big] dt\notag\\
								&\quad +\int_0^T\Big[\underbrace{(p_h,\Div \PUh v)}_{(viii)}-\underbrace{(f_h,v-\PUh v)}_{(ix)}\Big] dt,\notag
							\end{align}
						where $\PUh$ is the $L^2$-projection onto $\Uh$ as before. We consider the limits $h,\Delta t\to 0$ in each of these terms. For the first term, we can rewrite the integrals as follows:
\begin{align*}
							\int_0^T (i) dt & = \frac{1}{2\Delta t}\int_{\Delta t}^T (3u_h(t)-4u_h(t-\Delta t)+ u_h(t-2\Delta t),v(t)) dt + \frac{1}{\Delta t}\int_0^{\Delta t}(u_h(t)-u_h(t-\Delta t),v(t)) dt \\
							& =  \frac{1}{2\Delta t}\int_{2\Delta t}^T (3u_h(t)-4u_h(t-\Delta t)+ u_h(t-2\Delta t),v(t)) dt \\
							& = \frac{1}{2\Delta t}\left( \int_{2\Delta t}^T 3 (u_h(t),v(t)) dt - \int_{\Delta t}^{T-\Delta t}4 (u_h(t),v(t+\Delta t)) dt +\int_{0}^{T-2\Delta t}(u_h(t),v(t+2\Delta t)) dt \right) \\
							&  = \frac{1}{2\Delta t}\left( \int_{0}^T 3 (u_h(t),v(t)) dt - \int_{0}^{T}4 (u_h(t),v(t+\Delta t)) dt +\int_{0}^{T}(u_h(t),v(t+2\Delta t)) dt \right),
\end{align*}
where we used that $v$ is compactly supported in time and zero on $[0,2\Delta t]$ and $[T-2\Delta t,T]$. Letting 
\begin{equation*}
	D_t^+ v(t) := \frac{-v(t+2\Delta t)+4v(t+\Delta t)-3 v(t)}{2\Delta t},
\end{equation*}
we can rewrite this as
\begin{equation*}
	\int_0^T (i) dt  = -\int_0^T (u_h(t),D_t^+ v(t)) dt. 
\end{equation*}
						Since $D_t^+ v(t)\to \partial_t v(t)$ pointwise due to the smoothness of $v$, and $u_h\weakstar u$ in $L^\infty(0,T;L^2(\dom))$, we obtain that the first term converges to
						\begin{equation*}
							\int_0^T (i) dt \stackrel{h\to 0}{\longrightarrow} -\int_0^T (u(t),\partial_t v(t)) dt .
						\end{equation*}
						 For the second term, we  use that $\hu_h,\tu_h\to u$ in $L^2([0,T]\times\dom)$ and $\hu_h,\tu_h\weak u$ in $L^2([0,T];H^1_0(\dom))$:
						\begin{multline*}
							\int_0^T (ii) dt =  \int_0^T\int_{\dom}(\hu_h\cdot \Grad)\tu_h\cdot v +\frac12\Div \hu_h \tu_h \cdot v dxdt\\
							\stackrel{h\to 0}{\longrightarrow}\int_0^T\int_{\dom} (u\cdot\Grad)u \cdot v + \frac12 \Div u u\cdot v dx dt =\int_0^T\int_{\dom} (u\cdot\Grad)u \cdot vdx dt,
						\end{multline*}
						using that $\Div u=0$ a.e. in $[0,T]\times\dom$.
						For the third term, we use that $\Grad \tu_h\weak \Grad u$ in $L^2([0,T]\times\dom)$ and hence
						\begin{equation*}
							\int_0^T (iii) dt \stackrel{h\to 0}{\longrightarrow}\mu \int_0^T (\Grad u,\Grad v) dt.
						\end{equation*}
						We consider the fourth term: We have
						\begin{equation*}
							\begin{split}
								\int_0^T (iv) dt &= \sum_{m=0}^{N-1}\int_{t^m}^{t^{m+1}} (\PUh f^{m+1},v) dt\\
								& = \sum_{m=0}^{N-1}\int_{t^m}^{t^{m+1}}(  f^{m+1},\PUh v) dt\\
								& = \sum_{m=0}^{N-1}\int_{t^m}^{t^{m+1}} \frac{1}{\Delta t}\int_{t^{m+1/2}}^{t^{m+3/2}}(  f(s),\PUh v(t)-v(t))ds dt + \sum_{m=0}^{N-1}\int_{t^m}^{t^{m+1}} \frac{1}{\Delta t}\int_{t^{m+1/2}}^{t^{m+3/2}}(  f(s), v(t)) ds dt\\
								& \stackrel{h,\Delta t\to 0}{\longrightarrow} \int_0^T (f,v) dt,
							\end{split}
						\end{equation*}
						since time averages converge in $L^2$ and $\norm{v-\PUh v}_{L^2(\dom)}\to 0$ as $h\to 0$. Hence the left hand side of~\eqref{eq:uapp} converges to
						\begin{equation*}
							\int_0^T\left[-( u ,\partial_t v)+ ((u\cdot \Grad) u, v)+\mu(\Grad u,\Grad v) -(f,v)\right] dt.
						\end{equation*}
						In order to conclude that $u$ satisfies the weak formulation of~\eqref{eq:NS}, we thus need to show that the right hand side converges to zero.
						So, we consider the terms on the right hand side: For term (v), since $D_t^- u_h\in \Uh$, we have $(D^-_t u_h,v)=(D^-_t u_h,\PUh v)$, hence this term vanishes.
						Next,
						\begin{equation*}
							\begin{split}
								\left|\int_0^T (vi)dt \right|&=\left|\int_0^Tb(\hu_h,\tu_h,v-\PUh v) dt\right|\\
								&\leq \left(\norm{\hu_h}_{L^\infty(0,T;L^2)}\norm{\Grad\tu_h}_{L^2([0,T]\times\dom)}+\frac12 \norm{\Div\hu_h}_{L^2([2\Delta t,T]\times\dom)}\norm{\tu_h}_{L^\infty(0,T;L^2)} \right)\norm{v-\PUh}_{L^2([0,T]\times\dom)}\\
								&\leq C h \norm{v }_{L^2(0,T;H^1(\dom))},
							\end{split}
						\end{equation*}
						where we used the approximation property of the $L^2$-projection,~\eqref{eq:L2approx}, and the energy estimate, Lemma~\ref{lem:discenergy}. Thus also this term vanishes as $h,\Delta t \to 0$.
						For term (vii), we have
						\begin{equation*}
							\begin{split}
								\left|\int_0^T (vii)dt \right|&=	\left|\mu\int_0^T(\Grad\tu_h,\Grad(v-\PUh v))dt\right|\\
								&\leq \mu \norm{\Grad\tu_h}_{L^2([0,T]\times\dom)}\norm{\Grad(v-\PUh v)}_{L^2([0,T]\times\dom)}\\
								&\leq C h\norm{v}_{L^2(0,T;H^2(\dom))},
							\end{split}
						\end{equation*}
						again using the approximation property of the $L^2$-projection,~\eqref{eq:H1approx}, and the discrete energy estimate. 
						For the term (viii), we have,
						\begin{equation*}
							\begin{split}
								\left|\int_0^T (viii) dt\right|& =	\left|\int_0^T(\Grad p_h,\PUh v-v) dt\right|\\
								&\leq  \norm{\Grad p_h}_{L^2([0,T]\times \dom)}\norm{\PUh v-v}_{L^2([0,T]\times\dom)}\\
								& \leq C\frac{h^{k+1}}{\Delta t} \norm{v}_{L^2([0,T]; H^{k+1}(\dom))},
							\end{split}
						\end{equation*}
						similar to the estimate~\eqref{eq:pressuretimecont}.  Clearly, also this term goes to zero if $h^{k+1}=o({\Delta t})$. 
						Finally, for the ninth term, we have, since $f_h=\PUh f^m$,
						\begin{equation*}
							(ix) = (f_h,v-\PUh v) = 0,
						\end{equation*}
						i.e., this term vanishes.
						Thus we obtain in the limit
						\begin{equation}
							\label{eq:uprelimlimit}
							\int_0^T\left[-( u ,\partial_t v)+ ((u\cdot \Grad)u,v)+\mu(\Grad u,\Grad v) -(f,v)\right] dt = 0,
						\end{equation}
						for any divergence free $v\in C^\infty_c([0,T)\times\dom))$.
						Next, we reconsider the discrete energy balance~\eqref{eq:prelimenergyestimate}. Adding $2\Delta t$ times~\eqref{eq:firststepenergy2} to it and using~\eqref{eq:u0estimate}, we obtain
						 	\begin{align}\label{eq:prelimlimitenergy}
						 		&E^M_h  	+\sum_{m=1}^{M-1}\norm{u^{m+1}_h-2u^m_h+ u^{m-1}_h}_{L^2}^2+ 3\sum_{m=1}^{M-1}\norm{\tu^{m+1}_h-u^{m+1}_h}_{L^2}^2 
						 		+ 4 \mu\Delta t \sum_{m=0}^{M-1}\norm{\Grad \tu^{m+1}_h}_{L^2}^2\\
						 	&\qquad 	= \norm{u_h^1}_{L^2}^2 + \norm{2 u_h^1 -u_h^0}_{L^2}^2 + \frac{4\Delta t}{3}\norm{\Grad p_h^1}_{L^2}^2+4 \Delta t\sum_{m=0}^{M-1}( f^{m+1},\tu^{m+1}_h)\notag\\
						 	&\qquad 	-2\norm{u^1_h}_{L^2}^2 - 2\norm{\tu^1_h-u^0_h}_{L^2}^2 + 2\norm{u^0_h}_{L^2}^2
						 		-2\Delta t^2\norm{\Grad p^1_h}_{L^2}^2 +2\Delta t^2\norm{\Grad p^0_h}_{L^2}^2\notag\\
						 	&\qquad 	\leq -\norm{u_h^1}_{L^2}^2 + \norm{2 u_h^1 -u_h^0}_{L^2}^2+4 \Delta t\sum_{m=0}^{M-1}( f^{m+1},\tu^{m+1}_h)
						 	- 2\norm{\tu^1_h-u^0_h}_{L^2}^2 + 2\norm{u_0}_{L^2}^2\notag\\
						 	& \qquad \leq -\norm{u_h^1}_{L^2}^2 + \norm{2 u_h^1 -u_h^0}_{L^2}^2+4 \Delta t\sum_{m=0}^{M-1}( f^{m+1},\tu^{m+1}_h)
						  + 2\norm{u_0}_{L^2}^2.\notag
						 	\end{align}
						 	Rewriting this in terms of the piecewise constant functions $u_h$, $\bar{u}_h$, $\tu_h$ and $\hu_h$ and
						 	passing to the limit on the left hand side and using the properties of weak convergence, we have for almost any $t\in [0,T]$
						 	\begin{multline}
						 		2\norm{u(t)}_{L^2}^2 + 4\mu \int_0^t \norm{\Grad u_h(s)}_{L^2}^2 ds\\
						 		\leq \lim_{h,\Delta t\to 0}\left(E_h(t)  	+\Delta t^{-1}\int_0^t\norm{u_h(s)-\bar{u}_h(s)}_{L^2}^2 ds + 3\Delta t^{-1}\int_0^t\norm{\tu_h(s)-u_h(s)}_{L^2}^2 ds 
						 		+ 4 \mu\int_0^t\norm{\Grad \tu_h(s)}_{L^2}^2ds \right).
						 	\end{multline}
						 	Let us consider the terms on the right hand side. For the term involving $f$, we have
						 	\begin{equation*}
						 		4\int_0^t (f_h,\tu_h) ds \stackrel{\Delta t,h\to 0}{\longrightarrow} 4 \int_0^t (f,u) ds,
						 	\end{equation*}
						 	using the strong convergence of $\tu_h$ and $f_h$ in $L^2([0,T]\times\dom)$. We rewrite the remaining terms as follows:
						 	\begin{equation*}
						 		-\norm{u_h^1}_{L^2}^2 + \norm{2 u_h^1 - u_h^0}_{L^2}^2 + 2 \norm{u_0}_{L^2}^2 = 3 \norm{u_h^1}_{L^2}^2 +\norm{u_h^0}_{L^2}^2 - 4(u_h^1,u_h^0)+ 2 \norm{u_0}_{L^2}^2. 
						 	\end{equation*}
						 	We have $u_h^0\to u_0$ strongly in $L^2(\dom)$ and $u_h^1\weak u_0$ in $L^2(\dom)$. Therefore
						 	\begin{equation*}
						 	\lim_{h,\Delta t\to 0}\left(	\norm{u_h^0}_{L^2}^2 - 4(u_h^1,u_h^0)+ 2 \norm{u_0}_{L^2}^2 \right) = -\norm{u_0}_{L^2}^2.
						 	\end{equation*}
						 	Moreover, passing to the limit in $\Delta t$ times equation~\eqref{eq:firststepenergy2}, we obtain
						 	\begin{equation*}
						 		\lim_{h,\Delta t\to 0}\norm{u_h^1}_{L^2}^2 \leq \norm{u_0}_{L^2}^2.
						 	\end{equation*}
						 	Combining these estimates, we obtain that the right hand side of~\eqref{eq:prelimlimitenergy} satisfies
						 	\begin{equation*}
						 		\lim_{h,\Delta t\to 0}\left( -\norm{u_h^1}_{L^2}^2 + \norm{2 u_h^1 -u_h^0}_{L^2}^2+4 \Delta t\sum_{m=0}^{M-1}( f^{m+1},\tu^{m+1}_h)
						 		+ 2\norm{u_0}_{L^2}^2\right) \leq 2 \norm{u_0}_{L^2}^2 + 4 \int_0^t (f,u)ds, 
 						 	\end{equation*}
						 	which implies that 
						 	 the limit $u$ satisfies for almost any $t\in [0,T]$,
						\begin{equation}\label{eq:energylimit}
						\norm{u(t)}_{L^2(\dom)}^2 +2\mu\int_0^t\norm{\Grad u(s)}_{L^2(\dom)}^2 ds \leq \norm{u_0}_{L^2(\dom)}^2 + 2\int_0^t(f(s),u(s)) ds,
						\end{equation}
						which is the energy inequality~\eqref{eq:energyineq}. This implies the spatial regularity required in~\eqref{eq:regularity}. To see that $u$ satisfies the time regularity in~\eqref{eq:regularity}, we note that due to the spatial regularity of $u$, we can use the density of smooth functions and extend the weak formulation~\eqref{eq:uprelimlimit} to functions $v\in L^4(0,T;H^1_{\Div}(\dom))$. Finally, we establish the continuity at zero. We write
						\begin{equation*}
							\norm{u(t)-u_0}_{L^2(\dom)}^2 = \norm{u(t)}_{L^2(\dom)}^2 + \norm{u_0}_{L^2(\dom)}^2 - 2 (u(t),u_0).
						\end{equation*}
						First, we note that $\partial_t u\in L^{4/3}(0,T;(H^1_{\Div}(\dom)^*))$ and $u\in L^\infty(0,T;L^2_{\Div}(\dom))$ imply that $u$ is weakly continuous in time with values in $L^2_{\Div}(\dom)$, i.e., the mapping
						\begin{equation*}
							t\mapsto (u(t),v),
						\end{equation*}
						is continuous for any $v\in L^2_{\Div}(\dom)$ (see for example Lemma II.5.9 in~\cite{Boyer2013} or~\cite{Temam1977}). This implies that
						\begin{equation*}
							\lim_{t\to 0} (u(t),u_0) = \norm{u_0}_{L^2(\dom)}^2.
						\end{equation*}
						Furthermore, from Fatou's lemma, we obtain that
						\begin{equation*}
							\norm{u_0}_{L^2(\dom)}^2 \leq \lim_{t\to 0}\norm{u(t)}_{L^2(\dom)}^2. 
						\end{equation*}
						On the other hand, sending $t\to 0$ in the energy inequality~\eqref{eq:energylimit}, we obtain that
						\begin{equation*}
							\lim_{t\to 0}\norm{u(t)}_{L^2(\dom)}^2 \leq \norm{u_0}_{L^2(\dom)}^2.
						\end{equation*}
						Thus
						\begin{equation*}
							\lim_{t\to 0} \norm{u(t)}_{L^2(\dom)}^2 = \norm{u_0}_{L^2(\dom)}^2 
						\end{equation*}
						and we conclude that
						\begin{equation*}
							\lim_{t\to 0}\norm{u(t)-u_0}_{L^2(\dom)}^2 = 0.
						\end{equation*}
						Therefore,   $u$ is a Leray-Hopf weak solution of~\eqref{eq:NS}. 
					\end{proof}
					
						\section*{Acknowledgments}
						I would like to thank Alexandre Chorin and Saleh Elmohamed for insightful discussions and providing useful references on this topic.
						
						\appendix
						\section{Discrete Gr\"onwall inequality}
						The following lemma from~\cite[Proposition 4.1]{Emmrich1999} is useful for proving an energy bound in the case that the external source in the fluid equation is nonzero.
						\begin{lemma}[\cite{Emmrich1999}]
							\label{lem:discretegronwall}
							Let $\{a_n\}_{n\in \N}\geq 0$, and $\{b_n\}_{n\in \N}\geq 0$ be two nonnegative sequences satisfying
							\begin{equation}\label{eq:assgronwall}
								a_{n+1} \leq b_{n+1} + \nu\Delta t\sum_{j=1}^{n+1} a_j,\quad n=0,1,2,\dots
							\end{equation}
							for parameters $\nu,\Delta t\geq 0$ and assume that $1-\nu\Delta t>0$.
							Then $a_{n+1}$ satisfies
							\begin{equation}\label{eq:discgronwall}
								a_{n+1}\leq b_{n+1} +\frac{\nu\Delta t}{1-\nu\Delta t}\sum_{j=0}^n\left(\frac{1}{1-\nu\Delta t}\right)^{n-j} b_{j+1}.
							\end{equation}
							Moreover, if $\{b_n\}_{n\in\N}$ is monotonically increasing, it follows
							\begin{equation*}
								a_n \leq b_n \left(\frac{1}{1-\nu\Delta t}\right)^n.
							\end{equation*}
						\end{lemma}
						\begin{proof}
							For the proof, see \cite[Proposition 4.1]{Emmrich1999}.
						\end{proof}

						\section{Aubin-Lions-Simon lemma}
						\label{app:ALSlemma}
						
						We recall Simon's version of the Aubin-Lions-Simon lemma~\cite[Theorem 1]{Simon1987}:
						\begin{theorem}[Simon's compactness criterion~\cite{Simon1987}]
							\label{thm:Simonlemma}
							Let $F\subset L^p(0,T;B)$ where $1\leq p<\infty$ and $B$ is a Banach space. Then $F$ is relatively compact in $L^p(0,T;B)$ if and only if the following two conditions are satisfied:
							\begin{enumerate}
								\item $\left\{\int_{t_1}^{t_2} f(t) dt \, : \, f\in F\right\}$ is relatively compact in $B$ for all $0<t_1<t_2<T$,  \label{cond1} 
								
								\item $\norm{f(\cdot + \tau)-f}_{L^p(0,T-\tau;B)}\longrightarrow 0$ as $\tau\to 0$, uniformly for $f\in F$.   \label{cond2}
							\end{enumerate}
							
						\end{theorem}

								\bibliographystyle{abbrv}
								\bibliography{projection}
								
							\end{document}